\documentclass[11pt,onecolumn]{article}
\usepackage{amsmath,amsthm,amsfonts}
\usepackage[bottom=1.3in]{geometry}             
\usepackage[numbers,sort]{natbib}
\usepackage{authblk}
\usepackage{graphicx}
\usepackage{fancyhdr}
\usepackage{amssymb}
\usepackage{amsmath}
\usepackage{epstopdf}
\usepackage[all,cmtip]{xy}
\usepackage[pagebackref,colorlinks,
    linkcolor={red!70!black},
    citecolor={green!60!black},
    urlcolor={blue!80!black}]{hyperref}
\usepackage[capitalize]{cleveref}
\usepackage{stmaryrd}
\usepackage{stackengine}
\usepackage{enumitem}
\usepackage{mathtools}
\usepackage{subcaption}
\usepackage{doi}
\usepackage{tikz}
\usepackage[textsize=footnotesize]{todonotes}
\usepackage{mathrsfs}
\usepackage{backref}
\usepackage[toctitles]{titlesec}

\newcommand\barbelow[1]{\stackunder[1.5pt]{$\hspace{2pt}#1\hspace{2pt}$}{\rule{1ex}{0.075ex}}}
\newcommand\barbeloww[1]{\stackunder[1.5pt]{$#1$}{\rule{1ex}{0.075ex}}}
\newcommand{\barabove}[1]{\hspace{2pt}\bar{#1}\hspace{2pt}}
\newcommand{\ot}{\hspace{2pt} \check\otimes\hspace{2pt}}
\theoremstyle{plain}
\newtheorem{theorem}{Theorem}[section]
\newtheorem{lemma}[theorem]{Lemma}

\newtheorem{corollary}[theorem]{Corollary}
\newtheorem{proposition}[theorem]{Proposition}
\theoremstyle{remark}

\newtheorem{remark}[theorem]{Remark}
\theoremstyle{definition}
\newtheorem{definition}[theorem]{Definition}
\newtheorem{example}[theorem]{Example}

\AtBeginEnvironment{example}{%
  \pushQED{\qed}%
}
\AtEndEnvironment{example}{\popQED\endexample}
\AtBeginEnvironment{remark}{%
  \pushQED{\qed}%
}
\AtEndEnvironment{remark}{\popQED\endexample}

\AtBeginEnvironment{definition}{%
  \pushQED{\qed}%
}
\AtEndEnvironment{definition}{\popQED\endexample}

\usepackage{makecell}

\usepackage{arydshln}
\newcommand{\Rnn}{\mathbb{R}_{\geqslant 0}}
\newcommand{\mc}{\mathcal}
\newcommand{\R}{\mathbb{R}}
\newcommand{\N}{\mathbb{N}}
\newcommand{\C}{\mathbb{C}}

\newcommand{\omin}{\barbelow{\otimes}}
\newcommand{\omax}{\barabove{\otimes}}

\title{\textsc{Beyond Operator Systems}}

\author[1]{Gemma De les Coves}
\author[1]{Mirte van der Eyden\thanks{corresponding author, email: mirte.van-der-eyden@uibk.ac.at}}
\author[2]{Tim Netzer}
\affil[1]{\small Institute for Theoretical Physics, University of Innsbruck, Austria}
\affil[2]{\small Department of Mathematics, University of Innsbruck, Austria}

\date{\today}                   

\begin{document}

\pagestyle{fancy}
\fancyhf{}
\fancyhead[LE]{\nouppercase{\rightmark\hfill\leftmark}}
\fancyhead[RO]{\nouppercase{\leftmark\hfill\rightmark}}
\fancyfoot[C]{\hfill\thepage\hfill}
\setlength{\footskip}{1.5cm}
\setlength{\abovedisplayskip}{3pt}
\setlength{\belowdisplayskip}{3pt}
\maketitle
\begin{abstract}
Operator systems connect operator algebra, free semialgebraic geometry and quantum information theory. In this work we generalize operator systems and many of their theorems. While positive semidefinite matrices form the underlying structure of operator systems, our work shows that these can be promoted to far more general structures. For instance, we prove a general extension theorem which unifies the well-known homomorphism theorem, Riesz' extension theorem, Farkas' lemma and Arveson's extension theorem. On the other hand, the same theorem gives rise to new vector-valued extension theorems, even for invariant maps, when applied to other underlying structures. We also prove generalized versions of the Choi--Kraus representation, Choi--Effros theorem, duality of operator systems, factorizations of completely positive maps, and more, leading to new results even for operator systems themselves. In addition, our proofs are shorter and simpler, revealing the interplay between cones and tensor products, captured elegantly in terms of star autonomous categories. This perspective gives rise to new connections between group representations, mapping cones and topological quantum field theory, as they correspond to different instances of our framework and are thus siblings of operator systems. 
\end{abstract}

\begin{center}
\emph{A short video abstract can be found \href{https://youtu.be/t9LaNAPKoeE}{here}}
\end{center}

\tableofcontents
\section{Introduction}

Abstract operator systems (aos) are a fruitful framework to study objects from various areas in mathematics. Originating from operator algebra, particularly in the study of spaces of operators on Hilbert spaces, in free semialgebraic geometry they provide a dimension-free theory of sets of positive semidefinite matrices. The latter has led to a free version of Hilbert's 17th problem \cite{Helton02} and provided deep insights into various problems in convex optimization \cite{Fritz17, PlNe23}. 
But operator systems also cover important objects in quantum information theory, such as separable and entangled states, positive maps, quantum channels, quantum graphs or mapping cones\,---\,see for example \cite{DeNe21} and references therein.

The notion of a positive semidefinite (psd) matrix/operator is crucial to all of these concepts, and deeply anchored in the theory of abstract operator systems. Another important ingredient is the tensor product of convex cones; as will be explained later, an abstract operator system can be seen as a uniform choice of a tensor product with the cone of psd matrices (psd cone). 

We ask: 
\begin{quote}
Which facts about abstract operator systems are accidental to the psd cone and which hold for  general convex cones? 
\end{quote}
This question is relevant, for example, in efforts to single out quantum theory, where states and effects are modeled by elements in the psd cone.
General probabilistic theories shed light on the role of the psd cone by seeking to provide  alternatives to quantum theory, where states are described as elements in a convex cone, effects as elements in the dual cone, and composition as the tensor product of the corresponding cones \cite{Plavala_2023}. 
Tensor products of general cones are a source of unanswered questions for which the framework of abstract operator systems is too narrow \cite{Aubrun21,Aubrun2022,debruyn}. 

\begin{figure}[t]
    \centering
    \includegraphics{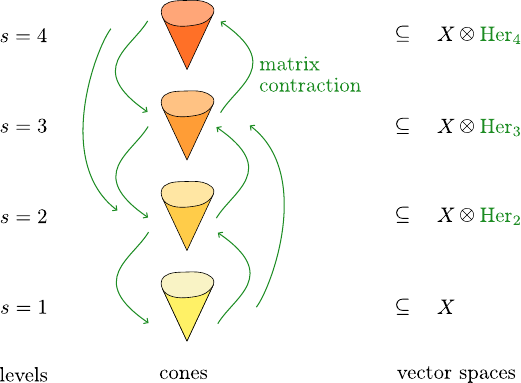}
    \caption{\small An abstract operator system on a real vector space $X$ consists of a convex cone inside $X \otimes \textrm{Her}_s$ for every $s$ ($\textrm{Her}_s$ are $s\times s$ complex Hermitian matrices), such that the cones are compatible under matrix contractions. These morphisms and underlying vector spaces, i.e. everything in green, will be formalized as a \emph{stem} in general conic systems. The yellow cone will be referred to as `the cone at the base level'. } 
    \label{fig:aos_def}
\end{figure}

In this paper, we present a framework that generalizes abstract operator systems to so-called \emph{general conic systems}. Towards explaining the framework, let us informally introduce the existing setting. An abstract operator system over a real vector space $X$ consists of a convex cone $\mc{C}(s)$ inside the space $\textrm{Her}_{s}(X)$ of Hermitian $s\times s$ matrices with entries in $X$, for every $s \in \N$ (\cref{fig:aos_def}). The cones $\mc{C}(s)$ can be chosen at will, as long as a compatibility is ensured, by which elements in the cone at level $s$ end up in the cone at level $t$ when contracting with $s \times t$ matrices. What is not obvious at first sight is that it is precisely this consistency condition that makes the psd cone play an important role here. For a general conic system, we can choose another underlying structure, called \emph{stem}, which replaces the levels of Hermitian matrices and the matrix contractions (everything green in \cref{fig:aos_def}). Instead, we can choose other vector spaces and other relations, and consequently also change the underlying psd cone to another convex cone. After fixing the stem, a general conic system over a real vector space $X$ consists of a convex cone inside $X \otimes V$ for every vector space $V$ in the stem, such that the cones are compatible under the chosen relations. 

The main finding of this work is that most of the interesting properties of abstract operator systems are in fact not due to the specific underlying structure of psd cones, but can be proven for far more general structures, i.e. for any stem. In this sense, we feel that we unlock the power of abstract operator systems theory and make it fruitful for other areas. We show that central theorems on abstract operator systems, such as the Choi--Kraus representation (\cref{cor:chkr}), the Choi--Effros realization theorem, the Effros--Winkler separation theorem (combined in \cref{cor:sep}), Arveson's extension theorem (\cref{thm:ext}), the Choi's theorem (\cref{thm:choi}), the existence of a minimal and maximal abstract operator systems (\cref{thm:minmax}), a duality theorem (\cref{thm:dual}) and the classification of systems with a finite-dimensional realization (\cref{thm:poly}), remain true for general conic systems. 
In fact, their proofs are shorter and simpler in the more general setting. 

Since we have many stems in our garden, we find that these generalized theorems cover new, but also existing results in other areas. For example, the extension theorem is the root of Arveson's extension theorem, the homomorphism theorem of vector spaces, Riesz extension theorem and Farkas' lemma, all with a common and simple proof (\cref{ex:theorems}, \cref{cor:riesz}). When applied to other stems, this statement gives new results, a vector valued Riesz extension theorem (\cref{cor:riesz}), or  extension theorems for invariant maps (\cref{cor:ext_opreps}, \cref{cor:ext_conereps}), for example.

What is the structure of the stems in general conic systems? The requirements can be formulated elegantly in the language of category theory. A stem is a functor from a source category to the category of finite-dimensional vector spaces that `picks out' the vector spaces that form the levels,  as well as the subsets of linear maps that give the consistency relations  (\cref{fig:stem}). The source category needs to be star autonomous in order to ensure compatibility with dual spaces and tensor products, and the stem functor needs to preserve this structure. Examples of star autonomous categories are the category of finite-dimensional vector spaces itself, the category of simplex cones, the category of all convex cones, and the category of finite-dimensional Hilbert spaces. But representations of groups on cones, unitary group representations and cobordisms are also star autonomous categories, and functors from these categories to the category of finite-dimensional vector spaces are valid stems in our framework. In case of the cobordism category, such functors are known as topological quantum field theories \cite{Atiyah1988}. We define general conic systems on top of all of these stems.

\begin{figure}
    \centering
    \includegraphics{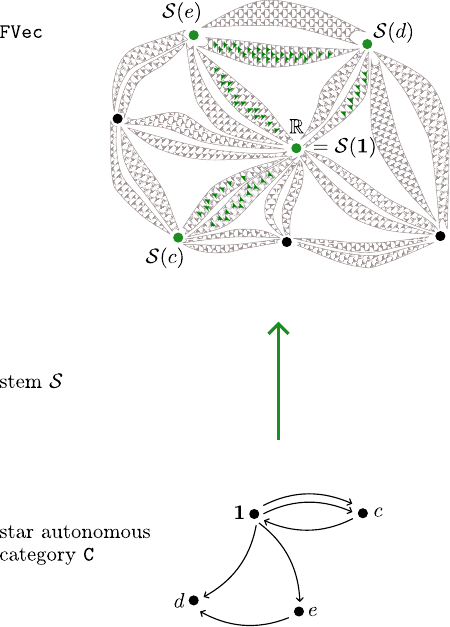}
    \caption{\small  A category consists of objects (black dots) and morphisms (black arrows in {\tt C}, grey directional shapes in {\tt FVec}). A stem $\mc{S}$ is a functor from a star autonomous category ${\tt C}$ to the category of finite-dimensional vector spaces ${\tt FVec}$ which picks out the vector spaces and the linear maps that will be used to build a general conic system, i.e., that play the role of everything green in \cref{fig:aos_def}. Here, the green dots and arrows in {\tt FVec} are in the image of the stem, whereas the black/grey are not. For example, for aos the vector spaces of Hermitian matrices of all sizes and the linear maps given by matrix contractions would be green, i.e. in the image of $\mc{S}$. }
    \label{fig:stem}
\end{figure}

Finally, for every fixed stem we consider the category of all possible general conic systems, with completely positive maps as morphisms. We prove that every such category is again star autonomous, and can be used as the starting point of a new stem (\cref{thm:hierarchy}). This provides a systematic way of constructing new and previously unknown star autonomous categories.

For readers who are only interested in operator systems, our work contains the following new results. 
First, we formulate the previously called `free dual' of an operator system \cite{Berger21} in terms of completely positive maps, and observe that it equals the classical level-wise dual of the cones (\cref{prop:dual}, \cref{def:dual}, \cref{ex:dual}\ref{dual_ex_3}). These observations lead to a strong duality theory that results in simpler proofs of the theorems mentioned above (\cref{ssec:theorems}). Second, we characterize completely positive maps between two operator systems in terms of maximal tensor products of these operator systems at level one (\cref{thm:cpdual}), which was only known for positive maps between single cones so far. Similarly, the completely positive linear maps that admit a completely positive factorization through a matrix space are given by the minimal tensor product of the operator systems at level one. Finally, we show that the category of operator systems is star autonomous (\cref{thm:hierarchy}). 

This paper is structured as follows. \cref{sec:prelim} contains the preliminaries, where we introduce cones and tensor products, summarize the most important results on abstract operator systems and present the required categorical notions. In \cref{sec:theory} we define stems and general conic systems, provide three running examples, revisit and prove results about abstract operator systems in this more general setting, provide a strong duality theory and study tensor products of general conic systems. In \cref{sec:app} we provide more examples of general conic systems, namely systems of general cones, mapping cones, conic group systems, unitary group representations and topological quantum field theories, and apply the theorems from \cref{sec:theory}. In \cref{sec:outlook} we conclude and give an outlook towards possible applications of the framework.

\section{Preliminaries}\label{sec:prelim} 
In this section we provide the preliminaries on cones and tensor products (\cref{ssec:cones}), abstract operator systems (\cref{ssec:aos}) and star autonomous categories (\cref{ssec:star}). In \cref{ssec:tools} we introduce some tools that we will use recurrently in the main proofs of \cref{sec:theory}. 

We use the following notation. We let $n,s,t \in \N$ be natural numbers throughout. Let $\textrm{Mat}_{s,t}$ denote the space of matrices of size $s \times t$ with entries in the field $\C$ of complex numbers, and let $\textrm{Mat}_{s}:= \textrm{Mat}_{s,s}$. We denote by $\textrm{Her}_{s} \subseteq \textrm{Mat}_s$ the real vector space of Hermitian matrices. 

Throughout this article, \emph{all}  vector spaces will be finite-dimensional and real, often denoted by capital Roman letters such as $V,W,X,Y$ Linear maps between vector spaces, e.g. from $X$ to $Y$, are denoted by Greek letters such as $\psi, \phi, \Psi \in \mathrm{Lin}(X,Y)$. Convex cones are denoted by capital Roman letters such as $C,D,E$. Categories are denoted in typewriter font, e.g. ${\tt C}$, $ {\tt FVec}$, and their objects often by lower case letters such as $c,d \in {\tt C}$. Functors are denoted in calligraphic font, e.g. $\mc{S}, \mc{G}$. Finally, we denote the bounded and unitary operators on a Hilbert space $H$ by $\mathsf{B}(H)$ and $\mathsf U(H)$, respectively.

\subsection{Cones and tensor products}\label{ssec:cones}
A convex cone $C \subseteq V$ is a subset of a real vector space such that for all $\lambda, \gamma \in \Rnn$ and $x,y \in C$ we have $$\lambda x + \gamma y\in C.$$ We will often omit the word `convex', because all our cones are convex.
A cone $C$ is \emph{generated} by a subset $B$ of $V$ if $C$ is the smallest convex cone containing $B.$ That is, $C$ consists of all nonnegative linear combinations of elements from $B$. In this case we write $C = \textrm{cone}(B)$, or $C=\overline{\textrm{cone}}(B)$ for the (topologically) closed conic hull. 
A cone is \emph{sharp} (also \emph{pointed} or \emph{salient} in the literature) if $C \cap -C = \{0\},$ and \emph{full} if it has nonempty interior, i.e. $C-C=V$. A cone that is closed, sharp and full is called \emph{proper}.

The \emph{dual} $C^\vee$ of a  cone $C \in V$ is defined as 
$$C^\vee\coloneq \{\varphi \in V' \mid \varphi(c) \geqslant 0 ~\forall c \in C \}, $$
where $V'= \mathrm{Lin}(V,\R)$ is the dual vector space, consisting of all linear maps from $V$ to $\R$. The dual of a cone is always closed. It is well-known that duals of sharp cones are full, and duals of full cones are sharp (see for example \cite{barv} for a classic introduction to convex cones).

An \emph{extreme ray} $R$ of a convex cone $C$ is a ray, i.e.\ a subset of the form $$R = \{\lambda v \mid \lambda \in \Rnn\} \subseteq C$$ for some $v \in C,$ such that for any $x,y \in C$ $$x+y \in R \implies x,y \in R.$$ In words, vectors in the extreme ray cannot be obtained as a nontrivial conic combination of elements in the cone. 

\begin{figure}[t]
    \centering
    \includegraphics{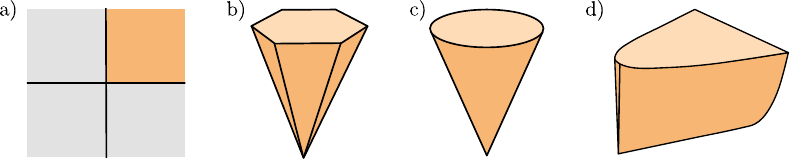}
    \caption{\small a) A simplex cone in $\R^2$, b) a polyhedral cone in $\R^3$, c) the Lorentz cone in $\R^3$ and d) the psd cone $\textrm{Psd}_2 \subseteq \textrm{Her}_2 \cong \R^4$ are examples of convex cones (defined in \cref{ex:cones}).}
    \label{fig:cones_examples}
\end{figure}

\begin{example}[Examples and facts about important cones (\cref{fig:cones_examples})]\label{ex:cones} \quad
\begin{enumerate}[label=$(\roman*)$]
    \item A \emph{simplex cone} (\cref{fig:cones_examples}a) is a cone generated by a basis of the space. Every simplex cone $S \subseteq \R^d$ is isomorphic to the positive orthant in $\R^d$, and thus in particular proper. The dual of a simplex cone is again a simplex cone. 
    \item A \emph{polyhedral cone} (\cref{fig:cones_examples}b) is a cone generated by a finite set, i.e., it has finitely many extreme rays. Alternatively, it is the intersection of finitely many half-spaces. Polyhedral cones are always closed. The dual of a polyhedral cone is again polyhedral, where the extreme rays of the original cone correspond to the half-spaces `carving out' the dual cone and vice versa. Every simplex cone is clearly polyhedral. 
    \item The \emph{Lorentz cone} (\cref{fig:cones_examples}c), or second order cone, defined as $$L_d = \{(t,x) \mid t \in \Rnn, x = (x_1, \ldots, x_d) \in \R^d, t^2 - \sum_{i=1}^d  x_i^2 \geqslant 0 \} \subseteq \R^{d+1},$$ is a cone with infinitely many extreme rays. Under the isomorphism between $\R^{d+1}$ and $(\R^{d+1})'$ provided by the standard inner product, this cone coincides with its dual.
    \item The \emph{psd cone} ${\rm Psd}_d \subseteq \textrm{Her}_d$ (\cref{fig:cones_examples}d) is the cone of positive semidefinite matrices of size $d \times d$, that is, Hermitian matrices with non-negative eigenvalues. It is a proper cone with infinitely many extreme rays which is self-dual under the trace inner product. This cone is of particular interest in quantum theory, since quantum states and effects are described by (normalized) elements from the psd cone.\qedhere
\end{enumerate}
\end{example}

The tensor product of two vector spaces $V$ and $W$ is uniquely defined and denoted, as usual, by $V \otimes W$. Yet, for two convex cones $C$ and $D$ there might exist many tensor products (\cref{fig:tensor_cones}). Two choices are particularly important: \begin{enumerate}[label=$(\roman*)$]
    \item  The \emph{minimal tensor product}, denoted $\omin$, defined as 
$$C \omin D := \left\{\sum_i c_i \otimes d_i \mid c_i \in C, d_i \in D \right\}.$$
    \item The \emph{maximal tensor product}, denoted $\omax$, defined as  
$$C \omax D := (C^\vee \omin D^\vee)^\vee,$$ using the natural identification between a space and its double-dual.
\end{enumerate}
Any other sensible choice of tensor product $\otimes$ fulfills $$C \omin D \subseteq C \otimes D \subseteq C \omax D.$$  In \cref{def:mon} we will present a tensor product on an arbitrary category.

\begin{figure}[t]
    \centering
    \includegraphics{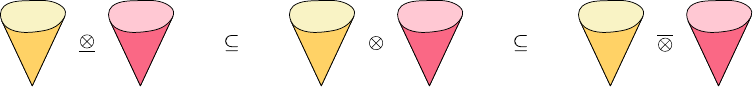}
    \caption{\small There are many possible tensor products of convex cones. For any two convex cones, their tensor product $\otimes$ is in between their minimal tensor product $\omin$ and maximal tensor product $\omax$.}
    \label{fig:tensor_cones}
\end{figure}

\begin{example}[Tensor products of some important cones] \quad
\begin{enumerate}[label=$(\roman*)$]
    \item For simplex cones, the minimal and maximal tensor product coincide, as can be easily seen for the positive orthant.
    \item  In  \cite{Aubrun21} it is proven that for two polyhedral cones $C_1$ and $C_2$, $C_1 \omin C_2 = C_1 \omax C_2$ if and only if at least one of the two cones is a simplex cone. In all other cases, these tensor products are different. 
    \item The minimal tensor product of two psd cones is known as the cone of \emph{separable} matrices, and the maximal as the cone of \emph{block-positive} matrices. Explicitly, 
    \begin{align*}\textrm{Psd}_d \omin \textrm{Psd}_s &= \left\{ \sum_i A_i\otimes B_i\mid A_i\in {\rm Psd}_d, B_i\in {\rm Psd}_s\right\}\\
    \textrm{Psd}_d \omax \textrm{Psd}_s &= \left\{M \in \textrm{Her}_{ds} \mid (v^* \otimes w^*)M(v \otimes w) \geqslant 0 ~\forall v \in \C^d, w \in \C^s \right\}.\qedhere\end{align*} 
\end{enumerate}
\end{example}

\subsection{Abstract operator systems}\label{ssec:aos}
In this section we review the definition of an abstract operator system as well as the results that will be generalized in \cref{sec:theory}. 

\begin{definition}[Abstract operator system (\cref{fig:aos_def})]\label{def:aos}
Let $X$ be a $\mathbb C$-vector space with involution and $X_{\rm her}$ its real subspace of Hermitian elements.
    An \emph{abstract operator system} (aos)  $\mc{C}$\footnote{We use the calligraphic font since we will redefine abstract operator systems as functors in \cref{def:acs}.} on $X$ consists of a proper cone $$\mc{C}(s) \subseteq X_{\rm her} \otimes \textrm{Her}_s=:{\rm Her}_s(X)$$ for every $s \geqslant 1$, such that for every $A \in \mc{C}(s)$ and $M \in \textrm{Mat}_{s,t}$ we have
    \begin{equation*}M^* AM \in \mc{C}(t).\qedhere\end{equation*}\end{definition}

Note that compared to  \cite{Paulsen03} we neglect the choice of an order unit $u \in \mc{C}(1)$ as  part of the definition. Order units are equivalent to interior points \cite{Cimpric11} and these always exist for full cones, and for our purpose there is no need to choose a specific one. 
\begin{example}[Examples of abstract operator systems]\label{ex:aos} 
If $X={\rm Mat}_d$ (and thus $X_{\rm her}={\rm Her}_d$) is itself a space of matrices, we obtain the important aos $\mathcal{PSD}_d$ on $X$ by setting $$\mathcal{PSD}_d(s):={\rm Psd}_{ds}\subseteq {\rm Her}_{ds}={\rm Her}_d\otimes{\rm Her}_s.$$ Other options are the aos of separable matrices $$\mathcal{SEP}_d(s):={\rm Psd}_d\omin {\rm Psd}_s$$ and the aos of block-positive matrices \begin{equation*}\mathcal{BP}_d(s):={\rm Psd}_d\omax {\rm Psd}_s.\qedhere\end{equation*}
\end{example}

\begin{figure}[t]
    \centering
    \includegraphics{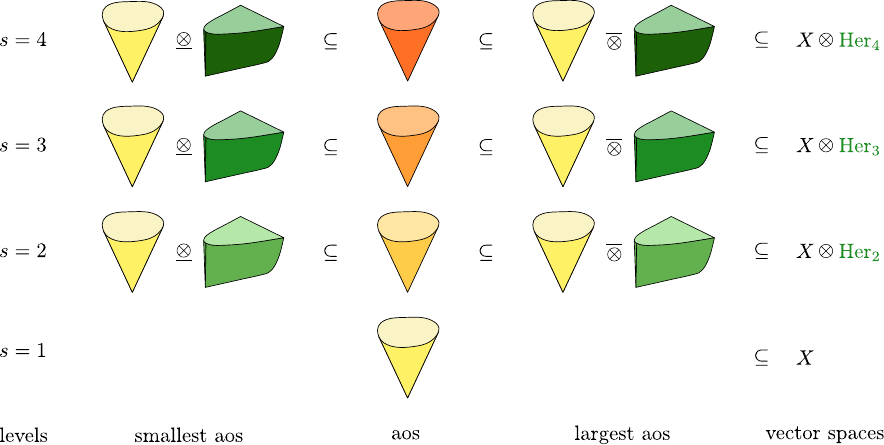}
    \caption{\small A given cone in $X$, here the yellow cone (base cone), can be extended to an abstract operator system by choosing a cone in every other level in a compatible way (\cref{fig:aos_def}). Due to this compatibility constraint, there is a smallest and largest abstract operator system over the base cone, given by the minimal and maximal tensor product of the base cone and the psd cone (green) in any level. For general conic systems there are different vector spaces and different compatibility conditions (everything that is green in \cref{fig:aos_def}) and therefore different cones play the role of the psd cone.  
    }
    \label{fig:aos}
\end{figure}

Now, every cone $D \subseteq X$ gives rise to a smallest and largest aos with $D$ at level one (see \cref{fig:aos}) \cite{Fritz17, Paulsen_2010}. The smallest system arises by applying all matrix contractions on the elements of $D$ to all levels $s$, and taking the smallest cones containing the images. The largest system arises by taking all elements that are sent to $D$ under all matrix contractions to level one. These definitions turn out to be equal to the following definitions in terms of minimal and maximal tensor products. The smallest system is given by the minimal tensor product of $D$ and the psd cone at level $s$,
$$\underline{\mc{C}}_D(s) := D \omin {\rm Psd}_s,$$
and the largest system is given by the maximal tensor product
$$ \overline{\mc{C}}_D(s):= D \omax {\rm Psd}_s. $$
Any aos $\mc{C}$ with $\mc{C}(1) = D$ lies in between these two (\cref{fig:aos}), i.e.
$$\underline{\mc{C}}_D(s) \subseteq \mc{C}(s) \subseteq \overline{\mc{C}}_D(s).$$

A \emph{concrete operator system} is defined as a unital $*$-subspace $X\subseteq {\mathsf B}(H)$ of operators on some Hilbert space $H$. On each such subspace there exists a canonical aos structure. In fact, the embedding $$X_{\rm her}\otimes{\rm Her}_s\subseteq \mathsf{B}(H)_{\rm her}\otimes{\rm Her}_s=\mathsf{B}(H^s)_{\rm her}$$  can be used to intersect the cone ${\rm Psd}(H^s)$\footnote{By $H^s$ we denote the $s$-fold sum of $H$.} of positive semidefinite operators with $X_{\rm her}\otimes{\rm Her}_s,$ and this will define an aos on $X$.

The Choi--Effros theorem states that every aos actually arises this way \cite{ChoiEffros}. Namely,
for every aos $\mc{C}$ on $X$ there exists a Hilbert space $H$ and a $*$-linear map $\Psi: X \to \mathsf{B}(H)$ to the space of bounded operators on $H$, such that for all $s \geqslant 1$ and $A \in {\rm Her}_s(X)$ we have 
\begin{equation}\label{eq:aos_realization} A \in \mc{C}(s) \Leftrightarrow (\Psi \otimes \textrm{id}_s)(A) \in {\rm Psd}(H^s). \end{equation}
 In words, while the layers $\mc{C}(s)$ can be any cone, they always have a one-to-one correspondence to a positive semidefinite cone on a Hilbert space $H^s$. The map $\Psi$ is called a \emph{(concrete) realization} of $\mc{C}$ (\cref{fig:aos_realization}). An aos $\mc{C}$ is called \emph{finite-dimensional realizable} if there exists a realization with a finite-dimensional Hilbert space $H$. Finite-dimensional realizable aos are equivalent to \emph{free spectrahedra} \cite{PlNe23}, which play an important role in free semialgebraic geometry and semidefinite programming.  In \cite{Fritz17} it was shown that the largest abstract operator system over $D$ has a finite-dimensional realization if and only if the cone $D$ is polyhedral. In addition, if the cone $D$ is polyhedral, then the smallest system is finite-dimensional realizable if and only if $D$ is a simplex cone. 

\begin{figure}[t]
    \centering
    \includegraphics{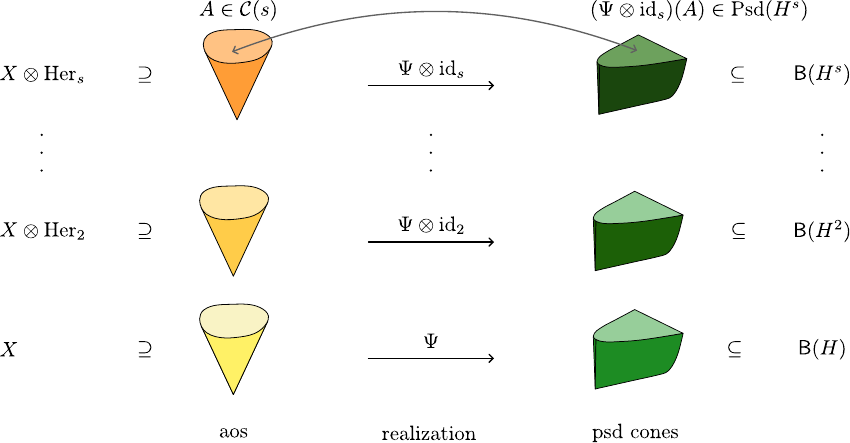}
    \caption{\small By the Choi--Effros theorem every abstract operator system (aos) has a (concrete) realization on a Hilbert space $H$. A realization is a map $\Psi: X \rightarrow \mathsf{B}(H)$ such that the cone at level $s$ is realized exactly by the psd cone in $\mathsf{B}(H^s)$ in the sense of Eq.\eqref{eq:aos_realization}.  }
    \label{fig:aos_realization}
\end{figure}

For each level $s$ and vector spaces $X,Y$, the linear map $\psi \in \mathrm{Lin}(X,Y)$ induces a linear map $$\psi \otimes \textrm{id}_s: X \otimes \textrm{Mat}_s \to Y \otimes \textrm{Mat}_s$$ that can have the following properties.
\begin{definition}[(Completely) Positive maps]\label{def:cp}
Let $\mc{C}$ be an aos on $X$ and $\mc{D}$ be an aos on $Y$. A linear map $\psi \in \mathrm{Lin}(X,Y)$ is 
\begin{enumerate}[label=$(\roman*)$]
    \item \emph{positive} if $\psi (\mc{C}(1)) \subseteq \mc{D}(1)$,
    \item \emph{$k$-positive} if $(\psi \otimes \textrm{id}_s)(\mc{C}(s)) \subseteq \mc{D}(s)$ for all $s \in \N$ such that $s \leqslant k$, and
    \item \emph{completely positive} if $(\psi \otimes \textrm{id}_s)(\mc{C}(s)) \subseteq \mc{D}(s)$ for all $s \in \N$.\qedhere
\end{enumerate}
\end{definition}

Note that completely positive maps between aos of psd matrices ($\mc{PSD}_d$ in \cref{ex:aos}) are precisely given by sums of compressions with matrices. This is the so-called Choi--Kraus representation of completely positive maps on matrix spaces.

Arveson's extension theorem is a statement about the extension of completely positive maps between concrete operator systems which is of great importance in operator theory. The finite-dimensional version can be stated as follows:

\begin{theorem}[Arveson's extension theorem \cite{Paulsen03}] \label{arveson}
Let $X\subseteq {\rm Mat}_d$ be a (concrete) operator system, and let $\psi\colon X \to {\rm Mat}_t$ be a completely positive map. Then $\psi$ can be extended to a completely positive map $\phi\colon {\rm Mat}_d \to {\rm Mat}_t$.
\end{theorem}

Finally, note that one can distinguish two layers in \cref{def:aos}, as summarized in the left-most column of \cref{tab:aos_to_gcs}. Namely, the choice of an aos consists of 
\begin{enumerate}
    \item a vector space $X$, and
    \item a cones $\mc{C}(s)$ in every level $X_{\rm her} \otimes \mathrm{Her}_s$.
\end{enumerate}
This part of the definition will stay the same for the generalized version, general conic systems. What will change, however, is the underlying structure, consisting of 
\begin{enumerate}[label=(\alph*)]
    \item the vector spaces that form the levels (which are $\mathrm{Her}_s$ for aos), and
    \item the linear maps between the levels under which the chosen cones need to be compatible (namely the compression with matrices in aos). 
\end{enumerate}
This underlying structure will be formalized by a stem and is colored green in \cref{fig:aos_def}.
In \cref{sec:theory}, all the above results (and more) will be proven for general conic systems, often with simpler proofs. 

\subsection{Star autonomous categories}\label{ssec:star}
The underlying structure for abstract operator systems can be phrased in terms of star autonomous categories and strong monoidal functors. We remark that, while the categorical notions help us to phrase all conditions in an elegant way, ultimately we are interested in the image of a functor inside the category of vector spaces, so all theorems and proofs in fact use techniques from linear algebra and convexity theory rather than category theory. 

Let us review the categorical concepts of star autonomous categories and strong monoidal functors, and present the most important categories and functors for this paper and their properties. We assume the definitions of a category, functor and natural transformation to be known, and refer to the excellent notes \cite{Perrone} if this is not the case. 

A category is monoidal if there is a notion of a tensor product on its objects. A good example is ${\tt FVec}$ where the objects are finite-dimensional real vector spaces and the morphisms are linear maps (see \cref{fig:FVec}).

\begin{figure}[th]
    \centering
    \includegraphics{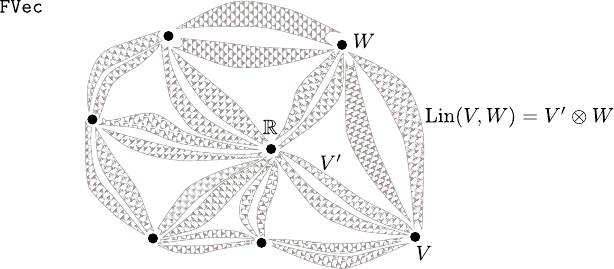}
    \caption{\small In the category of finite-dimensional vector spaces ${\tt FVec}$, each object is a finite-dimensional vector space, drawn as points, and morphisms are linear maps, represented by the grey shapes, where the morphisms follow the direction of the arrow points. {\tt FVec} is compact closed under the tensor product of vector spaces, with monoidal unit $\R$, dual objects $V'= \mathrm{Lin}(V,\R)$ and the relation $\mathrm{Lin}(V,W) = V'\otimes W$.}
    \label{fig:FVec}
\end{figure}

\begin{definition}[Symmetric monoidal category \cite{nlab:symmetric_monoidal_category}]\label{def:mon}
    A \emph{symmetric monoidal category} is a category ${\tt C}$ together with: 
    \begin{itemize} 
   \item  a functor $\otimes: {\tt C} \times {\tt C} \to {\tt C}$ called the \emph{tensor product},
   \item an object $\mathbf{1} \in {\tt C}$, called the \emph{monoidal unit}, 
   \item a natural isomorphism $\alpha_{a,b,c}: (a \otimes b) \otimes c \to a \otimes (b \otimes c)$ for all $a,b,c \in {\tt C}$, called the \emph{associator},
   \item a natural isomorphism $\lambda_a: \mathbf{1} \otimes a \to a$ for all $a \in {\tt C}$, called the \emph{left unitor},
   \item a natural isomorphism $\rho_a: a \otimes \mathbf{1} \to a$ for all $a \in {\tt C}$, called the \emph{right unitor}, and
   \item a natural isomorphism $\beta_{a,b}: a \otimes b \to b \otimes a$ for all $a,b \in {\tt C}$, called the \emph{braiding},
    \end{itemize}
    such that several diagrams commute (see \cite{nlab:symmetric_monoidal_category}), named the pentagon identity, triangle identity and hexagon identity, and such that $\beta_{b,a}\beta_{a,b} = {\rm id}_{a \otimes b}$. 
    A symmetric monoidal category will be denoted as a tuple $({\tt C}, \otimes, \mathbf{1})$.
\end{definition}
The diagrams impose that the associator and left and right unitors are compatible in the way that one (probably) expects, and they impose that the tensor product is symmetric, i.e. there is a natural isomorphism between $a \otimes b$ and $b \otimes a$, for all $a,b \in {\tt C}$. 

In some symmetric monoidal categories, every object has a dual object in a way that is compatible with the monoidal structure. These categories are called \emph{compact closed}. One can think again of ${\tt FVec}$, where every vector space $V$ has a dual vector space $V'$ and where $$(V\otimes W)'=V'\otimes W'.$$ If the duality is compatible with the tensor product only in a weaker way, the category is called \emph{star autonomous}. An interesting example is ${\tt FCone}$, the category of convex cones in finite-dimensional vector spaces, where objects are of the form $(C,V)$ (by slight abuse of notation often just denoted $C$), with $C\subseteq V$ a convex cone, and morphisms are positive linear maps between the cones,  $$\mathrm{Pos}(C,D):=\left\{\Phi \in \mathrm{Lin}(V,W) \mid  \Phi(C) \subseteq D\right\},$$
with $D \subseteq W$. In this category, every object $(C,V)$ has a dual, namely the dual cone inside the dual space, $(C^\vee, V')$. Using the minimal tensor product of cones, one can verify that ${\tt FCone}$ is symmetric monoidal, but it is not compact closed, as $$(C\omin D)^\vee= C^\vee\omax D^\vee,$$ i.e. the duality `swaps' the minimal and the maximal tensor product.  

\begin{definition}[Star autonomous category \cite{nlab:star_autonomous_category}]\label{def:star}
    A symmetric monoidal category ${\tt C}$ is \emph{star autonomous} if it is equipped with a full and faithful functor ${}^*: {\tt C}^{\textrm{op}} \to {\tt C}$ such that there is a natural bijection $${\tt C}(a\otimes b,c) \cong {\tt C}(a,(b\otimes c^*)^*)$$ for all $a,b,c\in{\tt C}
    $. A star autonomous category will be denoted as a tuple $({\tt C}, \otimes, \mathbf{1},{}^*)$.
\end{definition}
This definition imposes a certain compatibility between the morphism sets, the tensor product and the dual objects. Specifically, every star autonomous category is \emph{closed monoidal}, meaning that for any two objects $a,b \in {\tt C}$ there is another object $[a,b]$ called the \emph{internal hom}, that behaves as the object of morphisms between $a$ and $b$ and is compatible with the tensor product under the following natural isomorphism:
$$ {\tt C}(a \otimes b, c) \cong {\tt C}(a, [b,c]).$$

For vector spaces, this corresponds to the well-known fact that $\mathrm{Lin}(V,W)$ is a vector space itself. 
But more is true, namely the internal hom (that is, the vector space of linear maps) can be expressed in terms of the dual vector space, namely $\mathrm{Lin}(V,W) \cong V' \otimes W$. In a star autonomous category, the internal hom is given by 
\begin{equation}\label{eq:choi_jam}
   [a,b] = (a \otimes b^*)^*.
\end{equation}
Note that for vector spaces the tensor product fulfills  $(V \otimes W' )'=V' \otimes W$. As we saw earlier, if this self-duality of the tensor product holds, the category is compact closed, but it need not be fulfilled for star autonomous categories in general.

An example of a star autonomous category that is not compact closed is ${\tt FCone}$. For cones there is an identification between $\mathrm{Pos}(C,D)$ and an object in the category, sometimes referred to as the Choi--Jamio\l{}kowski isomorphism: 
\begin{equation} \mathrm{Pos}(C,D) = C^\vee \omax D = (C \omin D^\vee)^\vee.\end{equation}
Indeed, this corresponds to \eqref{eq:choi_jam}, and with this identity it can be shown that \cref{def:star} is fulfilled. 

Every star autonomous category is \emph{bimonoidal}, meaning that it is equipped with a second monoidal structure $\ot$ defined as the dual tensor product of $\otimes$: 
$$a \ot b \coloneq (a^* \otimes b^*)^*.$$ For ${\tt FVec}$ as well as any compact closed category, the two tensor products $\otimes$ and $\ot$ coincide. For ${\tt FCone}$, however, we see in \eqref{eq:choi_jam} that $\ot=\omax$, while $\otimes = \omin$. The dual tensor product has a corresponding monoidal unit, namely $\mathbf{1}^*$, and fulfills all the requirements of \cref{def:mon}. In terms of $\ot$, the condition of star autonomy can be written as $${\tt C}(a \otimes b, c) \cong {\tt C}(a, b^* \ot c).$$

In our framework we will be interested in functors that preserve the above properties.
\begin{definition}[Strong monoidal functor \cite{nlab:monoidal_functor}]
    Let $({\tt C}, \otimes_{{\tt C}} ,\mathbf{1}_{{\tt C}}) $ and $({\tt D}, \otimes_{{\tt D}},\mathbf{1}_{{\tt D}}) $ be two monoidal categories. A functor $\mc{F}: {\tt C} \to {\tt D}$ is called \emph{strong monoidal} if there is a natural isomorphism 
    $$ \mathbf{1}_{\tt D} \cong \mc{F}(\mathbf{1}_{\tt C}), $$
    and  natural isomorphisms
    $$ \mc{F}(a) \otimes_{{\tt D}} \mc{F}(b) \cong \mc{F}(a \otimes_{{\tt C}} b),$$
    such that several diagrams commute (see \cite{nlab:monoidal_functor}), ensuring unitality and associativity.
\end{definition}
To preserve the star autonomous structure we need a strong monoidal functor that commutes with the dual, i.e.\ we need natural isomorphisms $\mc{F}(a^*) \cong \mc{F}(a)^*$, resulting in a \emph{star monoidal functor}. 
For such a functor we obtain natural isomorphisms $${\tt D}(\mathcal F(a\otimes b),\mathcal F(c))\cong{\tt D}(\mathcal F(a),\mathcal F((b\otimes c^*)^*)).$$ Finally, we call a star monoidal functor $\mc{F}$ \emph{star autonomous} if the following diagram commutes:

$$\xymatrix{{\tt C}(a\otimes b,c)\ar@{}[r]|{\hspace{-0.35cm}\cong} \ar[d]_{\mathcal F}& {\tt C}(a,(b\otimes c^*)^*) \ar[d]^{\mathcal F}\\ {\tt D}(\mathcal F(a\otimes b),\mathcal F(c))\ar@{}[r]|{\hspace{-0.35cm}\cong}& {\tt D}(\mathcal F(a),\mathcal F((b\otimes c^*)^*))}$$
 
\begin{example}[Star autonomous functors]\label{ex:saf}\quad
\begin{enumerate}[label=$(\roman*)$]
\item\label{saf1} The category ${\tt FCone}$ is star autonomous and the category ${\tt FVec}$ is compact closed. The \emph{forgetful} functor 
\begin{align*}
    \mathcal F\colon {\tt FCone}&\to {\tt FVec} \\ (C,V)&\mapsto V \\ \Phi&\mapsto \Phi
\end{align*}
is star autonomous.
\item\label{saf2} The category ${\tt FHilb}$ with finite-dimensional complex  Hilbert spaces as objects and bounded linear operators as morphisms is compact closed. Since we want to restrict to covariant functors, we consider the opposite category ${\tt FHilb}^{\rm op}$ and obtain the star autonomous functor 
\begin{align*}
    \mathcal F\colon {\tt FHilb}^{\rm op}&\to {\tt FVec} \\ H&\mapsto \mathsf B(H)_{\rm her} \\ (T\colon H_2 \to H_1) &\mapsto (\mc{F}(T)\colon \mathsf{B}(H_1)_{\rm her} \to \mathsf{B}(H_2)_{\rm her}) \\ & \hspace{2cm} A \mapsto T^* A T .
\end{align*} 
\item\label{saf3} The \emph{innocent bystander} functor, also known as tensor functor, \begin{align*}
   X \otimes\cdot\colon {\tt FVec}&\to{\tt FVec} \\ V&\mapsto X \otimes V \\ \Phi&\mapsto \textrm{id}_X \otimes \Phi
\end{align*} 
 is \emph{not} star autonomous (not even monoidal), but will be important in \cref{ssec:acs}.\qedhere
\end{enumerate}
\end{example}

\subsection{Useful tools}\label{ssec:tools}
Let us now provide several preparatory results that will be helpful throughout the paper. 
An important tool is the isomorphism between linear maps, tensors and functionals on vector spaces
\begin{equation}
    \mathrm{Lin}(V,W) \cong V' \otimes W \cong \mathrm{Lin}(V\otimes W', \R), \label{eq:isom}
\end{equation}
by which given two vector spaces $V,W$ and a linear map $\phi\colon V \rightarrow W$, we can interpret $\phi$ as a tensor $T_\phi \in V' \otimes W$ or a functional $F_\phi: V \otimes W' \rightarrow \R$. 
One way to  make this  explicit is to write
\begin{align*}
 \phi: V \to W ; v \mapsto \sum_i \varphi_i(v) w_i
\end{align*} with $\varphi_i \in V', w_i \in W,$ and obtain \begin{align*} T_\phi&=\sum_i \varphi_i\otimes w_i\in V'\otimes W, \\
F_{\phi}(v\otimes \psi)&=\psi(\phi(v))=\sum_i\varphi_i(v)\psi(w_i).
\end{align*}
This is exactly why the category ${\tt FVec}$ is compact closed. By slight abuse of notation, we will denote all three interpretations with the same letter, e.g.\ $\phi$, from now on; from the context it will be clear which interpretation is meant. 

The following is a consequence of these multiple interpretations. 
\begin{lemma}[Linear maps and tensors] 
Given $\phi\in{\rm Lin}(V,W)\cong V'\otimes W$ and $\psi \in V \otimes X \cong {\rm Lin}(V',X)$, we have the following identity in $W\otimes X\cong X\otimes W$: 
$$(\phi \otimes {\rm id}_X) (\psi) = (\psi \otimes {\rm id}_W)(\phi). $$
\end{lemma}
\begin{proof}
Let $v_1,\ldots, v_n$ be a basis of $V$ and $v_1',\ldots, v_n'$ its dual basis, which is a basis of $V'$. Write $\phi = \sum_i v_i' \otimes w_i$ and $\psi = \sum_j v_j \otimes x_j$. 
Then $$(\phi \otimes {\rm id}_X) (\psi) = \sum_{i,j} v_i'(v_j) w_i \otimes x_j = \sum_i w_i \otimes x_i$$ and 
$$(\psi \otimes {\rm id}_W)(\phi) = \sum_{i,j} v_j(v_i') x_j \otimes w_i = \sum_i x_i \otimes w_i.$$
Under symmetry of the tensor product these expressions are the same. 
\end{proof}

For every $V \in {\tt FVec}$, the \emph{maximally entangled state} is an element in $V \otimes V'$:
$$m_V:= \sum_{i=1}^d v_i \otimes v'_i$$
where  $v_1,\ldots, v_d$ is a basis of $V$, $v_1',\ldots, v_d'$ its dual basis, and $d = \textrm{dim}(V)$. Note that the element $m_V$ is independent of the choice of basis.
Under the isomorphisms of Eq.\eqref{eq:isom}, we can also view $m_V$ in the following ways: 
\begin{enumerate}[label=$(\roman*)$]
    \item as the identity map ${\rm id}_{V'}\colon  V' \rightarrow V'; v \mapsto v$, 
    \item as the evaluation functional $\textrm{ev}_{V}\colon V' \otimes V \to \R; \varphi \otimes v \mapsto \varphi(v) $,
    \item as the linear map ${m}_V\colon \R \to V \otimes V'; 1 \mapsto m_V$.
\end{enumerate}

In particular, we can prove the following. 
\begin{lemma}[Maximally entangled state] 
Let $D\subseteq V$ be a convex cone, then $$m_V \in D \omax D^\vee .$$
\end{lemma}

\begin{proof}
Under the isomorphism in \eqref{eq:isom}, $m_V = \textrm{id}_{V'}\colon V' \rightarrow V' \in \mathrm{Pos}(D^\vee,D^\vee) \cong D \omax D^\vee$.
\end{proof}

\begin{example}[Maximally entangled state]\label{ex:max_ent}
Let $V=\textrm{Her}_d$, and let  $E_{ij}$ be the  $d\times d$ matrix with zeros except for a one in the $i,j$-th entry. One checks that  $$m_{\textrm{Her}_d}= \sum_{i,j=1}^d E_{ij} \otimes E_{ij} = \sum_{i,j=1}^d e_ie_j^* \otimes e_i e^*_j = \left(\sum_i e_i\otimes e_i\right)\left(\sum_j e_j \otimes e_j\right)^*.$$
In quantum theory this expression is known as the \emph{maximally entangled state} of two $d$-dimensional quantum systems.
\end{example}


\section[Theory]{Theory of general conic systems}\label{sec:theory}
We are now ready to present the results of our paper. We start by introducing \emph{stems} (\cref{ssec:index}), which form the underlying structure upon which general conic systems are built. The \emph{operator stem} will lead to the familiar notion of abstract operator systems, but is by far not the only interesting stem. We will define \emph{general conic systems} in \cref{ssec:acs} and provide some basic examples. In \cref{ssec:duality} we will present a duality theorem that is the basis of most of the forthcoming results. In \cref{ssec:theorems} we will revisit known theorems about abstract operator systems and generalize them to general conic systems. Finally, in \cref{ssec:tensorproducts}, we will introduce tensor products between general conic systems and study their properties, culminating in \cref{thm:hierarchy} which provides a hierarchy of stems, and is thus a source of many new examples. 

\subsection{Stems --- the underlying structures}\label{ssec:index}

 The underlying structure of abstract operator systems consists of the levels of Hermitian matrices of growing size, as well as the maps --- contractions with rectangular matrices --- between these levels. The relevant properties of this structure can be phrased in terms of a star autonomous category and a star autonomous functor to the category of finite-dimensional vector spaces ${\tt FVec}$, see \cref{ssec:star}. Every such choice of category and functor is called a \emph{stem} (\cref{fig:stem}).
 
\begin{definition}[Stem]\label{def:stem} Let ${\tt C}$ be a star autonomous category. A \emph{stem} $\mc{S}$ is a star autonomous functor \begin{equation*}\mathcal S\colon {\tt C}\to{\tt FVec}.\qedhere\end{equation*}
\end{definition} 
The objects of ${\tt C}$ serve as labels for the levels of the general conic system, and the stem $\mc{S}$ assigns a vector space to each level. For an object $c \in {\tt C}$ we refer to $\mc{S}(c)$ as the \emph{vector space at level} $c$, but note that the $c$'s generally do not form a classical `level structure' as they need not be ordered. The stem will also determine, for every two levels, the subset of linear maps between these spaces/levels. 
We recommend the reader to look at \cref{tab:aos_to_gcs} to see how the original levels of Hermitian matrices underlying an abstract operator system can be obtained from the \emph{operator stem} $$\mc{S}\colon {\tt FHilb^{op}} \to {\tt FVec}$$ of \cref{ex:saf}\ref{saf2}. 

For every stem $\mc{S}$, the vector spaces in its image are equipped with two convex cones which we call the intrinsic and cointrinsic cone (\cref{fig:bc_cones}). These cones play the role that the psd cone does for operator systems. While these two cones coincide in the case of operator systems, they can be distinct in general conic systems.  

In words, the intrinsic cone $A(c)$ is the subset of elements in the vector space $\mc{S}(c)$ that are positive when mapped down to the base level. Explicitly, we take all vectors in $\mc{S}(c)$ and apply all linear maps in the image of the stem from level $c$ to level $\mathbf{1}^*$. The vectors sent to a positive element in $\mc{S}(\mathbf{1}^*) = \R'$ under all those maps form a cone, denoted $A(c)$, as `coming from Above'. 
On the other hand, the cointrinsic cone $B(c)$ arises from maps in the other direction, namely by starting at level $\mathbf{1}$ and applying all maps that end up in level $c$. Explicitly, we apply all linear maps in the image of the stem from level $\mathbf{1}$ to level $c$ to the vector $1 \in \mc{S}(\mathbf{1})= \R$. We denote the closed conic hull of the vectors in the image by $B(c)$, as in `coming from Below'.

\begin{definition}[(Co)Intrinsic cones]\label{def:bc_cones}
Given a stem $\mathcal S$, for $c\in{\tt C}$ we define two closed convex cones in $\mathcal S(c)$:
\begin{align*}
A(c)&:=\{ v\in\mathcal S(c)\mid \mathcal S(\varphi)(v)\geqslant 0\  \forall \varphi\in{\tt C}(c,\mathbf 1^*)\},\\
B(c)&:=\overline{\rm cone}\{\mathcal S(\Phi)(1)\mid\Phi\in{\tt C}(\mathbf 1,c)\}.
\end{align*} 
We call $A(c)$ the \emph{intrinsic cone} at level $c$, and $B(c)$ the \emph{cointrinsic cone} at level $c$. 
\end{definition}
The intrinsic cone $A(c)$ is receives this name because its role is more important than that of $B(c)$, as we will see in \cref{ssec:duality}. 

\begin{remark}
Note that $A(\mathbf 1^*)\subseteq \mathbb R_{\geqslant 0}^\vee$ and $\mathbb R_{\geqslant 0} \subseteq B(\mathbf 1)$.
We call a stem \emph{self-dual} if 
$A(c)=B(c)$ for all $c\in{\tt C}$, and \emph{compact closed} if {\tt C} is a compact closed category (i.e. $(c\otimes d)^*$ is naturally isomorphic to $c^*\otimes d^*$, that is, $c\otimes d =  c \ot d$ for all $c,d\in{\tt C}$).
\end{remark}

\begin{figure}[t]
    \centering
    \includegraphics{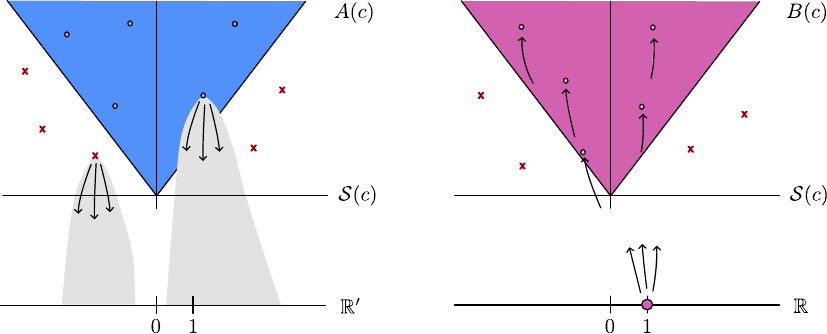}
    \caption{\small The intrinsic cone $A(c)$ on the left, and  the cointrinsic cone $B(c)$ on the right, as given in \cref{def:bc_cones}. The intrinsic cone consists of all vectors in $\mc{S}(c)$ that are nonnegative under all linear maps in $\mc{S}({\tt C}(c, \mathbf{1}^*))$. The red crosses are vectors for which at least one map sends them to a negative element in $\R'$. The cointrinsic cone at level $c$ is the closed conic hull of the image of the number $1 \in \R$ under all linear maps in $\mc{S}({\tt C}(\mathbf{1},c))$. The red crosses correspond to vectors that are not in this closed conic hull.}
    \label{fig:bc_cones}
\end{figure}

\begin{table}[t]
    \centering \small
    \begin{tabular}{l|l|l|l}
        & \thead{Operator System \\ (original)} &  \thead{Operator System \\ (categorical)} &  \thead{General Conic \\ System}\\
        \hline
       Level  & \makecell[l]{$s\in \N$} & \makecell[l]{$H \in {\tt FHilb}^{\rm op}$} & \makecell*[{{p{5cm}}}]{$c \in {\tt C}$, with ${\tt C}$ a  star autonomous category} \\
\hline 
       \makecell[l]{Level \\ spaces} & \makecell[l]{$\textrm{Her}_s$} & \makecell[l]{$\mathsf{B}(H)_{\rm her}$} & \makecell*[{{p{5cm}}}]{$\mc{S}(c)$,  with $\mc{S}\colon{\tt C} \to {\tt FVec}$ a star autonomous functor}\\
       \hline
       \makecell[l]{Maps \\ between\\ levels}& \makecell*[{{p{3.5cm}}}]{$\begin{aligned}
           \mathrm{Her}_s \to& \mathrm{Her}_t \\A \mapsto& V^*A V \end{aligned}$ \\ for $ V \in \textrm{Mat}_{s,t}$ }  
       & \makecell*[{{p{3.8cm}}}]{$\begin{aligned}\mathsf B(H_1)_{\rm her}\to& \mathsf B(H_2)_{\rm her} \\
       A \mapsto& T^*A T \end{aligned}$ \\ 
       for $T\in {\tt FHilb}^{\rm op}(H_1,H_2)$} & 
       \makecell*[{{p{5cm}}}]{$\mc{S}(\Phi)\colon \mc{S}(c_1) \to \mc{S}(c_2)$  \\ for $\Phi \in {\tt C}(c_1,c_2)$}\\
       \hline
       \makecell[l]{System \\ on $X$} & \makecell[l]{$\mc{C}(s) \subseteq X \otimes \textrm{Her}_s$,~ $(s \in \N)$} & \makecell[l]{functor $\mc{G}\colon {\tt FHilb}^{\rm op} \to {\tt FCone}$ } &  \makecell[l]{functor $\mc{G}\colon {\tt C} \to {\tt FCone}$ }\\
       \makecell[l]{such \\that } & \makecell*[{{p{3.5cm}}}]{ $({\rm id}_X\otimes \Phi)(\mc{C}(s)) \subseteq \mc{C}(t)$ \\ for all maps $\Phi$  from level $s$ to $t$} & \makecell[l]{$$\xymatrix{ {\tt FVec}\ar[r]^{X\otimes\cdot } & {\tt FVec} \\ {\tt FHilb}^{\rm op} \ar[r]^{\mc{G}}\ar[u]_{\mathcal S}& {\tt FCone}\ar[u]^{\mc{F}}}$$ \\commutes  }& \makecell[l]{$$\xymatrix{{\tt FVec}\ar[r]^{X\otimes\cdot } & {\tt FVec} \\ {\tt C} \ar[r]^{\mc{G}}\ar[u]_{\mathcal S}& {\tt FCone}\ar[u]^{\mc{F}}}$$\\commutes  }\\ \end{tabular}
    \caption{\small Operator systems (first column) can be phrased categorically (second column) as an instance of general conic systems (third column). 
 The first column contains the original definition of abstract operator systems (\cref{def:aos}), the second column contains the equivalent categorical definition of abstract operator systems (\cref{ex:acs}\ref{os}), and the third column contains the definition of a general conic system (\cref{def:acs}).}\label{tab:aos_to_gcs}
\end{table}

\begin{example}[Examples of stems]\label{ex:indstr}The following will serve as running examples throughout this section. More examples can be found in \cref{sec:app}.
\begin{enumerate}[label=$(\roman*)$]
\item\label{FVec}\textbf{(Vector stem)} Let us take as {\tt C} the category {\tt FVec} with the usual tensor product (whose monoidal unit is $\mathbb R$),  and duality $$V^*:=V'=\mathrm{Lin}(V,\mathbb R).$$ This is a compact closed category (\cref{ssec:star}). As the stem $\mathcal S\colon{\tt FVec}\to{\tt FVec}$ we take the identity, and call it the \emph{vector stem}. Since $$A(V)=\{0\}\quad\mbox{and}\quad B(V)=V$$ for all $V\in{\tt FVec}$,  this stem is compact closed but not self-dual.

\item\label{SCone}\textbf{(Simplex stem)} Let us choose for {\tt C} the category of simplex cones, denoted ${\tt SCone}$, which is the full subcategory of ${\tt FCone}$ containing only those objects $(S,V)$ where $S$ is a simplex cone, i.e.\ the conic hull of a basis of $V$. Again we often just write $S$ for the object. With the duality $$(S,V)^*:=(S^\vee, V')$$ and the fact that minimal and maximal tensor products coincide for simplex cones, we obtain that {\tt SCone} is a compact closed category.
Choosing for the stem $\mathcal S\colon{\tt SCone}\to{\tt FVec}$ the forgetful functor that keeps $V$ and forgets $S$, called the \emph{simplex stem}, we obtain  $$A(S) = S = B(S).$$ So this is a compact closed self-dual stem.

\item\label{aos}\textbf{(Operator stem)} The framework of operator systems is recovered by taking as stem $\mathcal S\colon {\tt FHilb}^{\rm op}\to{\tt FVec}$ the functor from \cref{ex:saf}\ref{saf2}, called the \emph{operator stem}.  For $H \in {\tt FHilb}^{\rm op}$ we obtain $$A(H)={\rm Psd}(H)=B(H),$$ the cone of positive semidefinite operators inside $\mathsf B(H)_{\rm her}$, so  this is a compact closed self-dual stem. This highlights the role of the psd cone for operator systems. See also \cref{tab:aos_to_gcs} for a comparison between the operator stem and the original definition of abstract operator systems. \qedhere
\end{enumerate}
\end{example}
 
Since operator systems are based on a self-dual stem, it has so far been hidden that there is not only one interesting cone at every level, but two, namely  $A(c)$ and $B(c)$, that happen to coincide for the operator stem. The different roles played by these cones will become clear in the rest of this paper.  
 
The following lemma collects some results on $A(c)$ and $B(c)$ that hold for all stems. For example, we prove that the intrinsic and cointrinsic cone are dual to each other.
 \begin{lemma}[Properties of (co)intrinsic cones]\label{lem:cones}
Let $\mathcal S\colon {\tt C}\to{\tt FVec}$ be a stem, and let $c, d \in{\tt C}$.
\begin{enumerate}[label=$(\roman*)$]
\item\label{co1} $A(c)^\vee=B(c^*),$ and vice versa.
\item\label{co2}  $B(c)\subseteq A(c)$ for all $c\in {\tt C} \  \Leftrightarrow\ B(\mathbf 1)\subseteq A(\mathbf 1).$
\item\label{co3} For $\Phi\in{\tt C}(c,d)$ we have $\mathcal S(\Phi)\in{\rm Pos}(A(c),A(d))\cap {\rm Pos}(B(c),B(d)).$
\item\label{co5} $B(c \ot c^*)=A(c^*\otimes c)^\vee$ contains the maximally entangled state $m_{\mathcal S(c)}.$ The evaluation map $${\rm ev}_{\mc{S}(c)}\colon\mathcal S(c)'\otimes\mathcal S(c)=\mathcal S(c^*\otimes c)\to\mathbb R$$ is nonnegative on $A(c^*\otimes c).$
\end{enumerate}
\end{lemma}
\begin{proof}\ref{co1} is clear from ${\tt C}(c,\mathbf 1^*)\cong{\tt C}(\mathbf 1,c^*)$ and the biduality theorem. For \ref{co2} assume $B(\mathbf 1)\subseteq A(\mathbf 1)$ and  take $\Phi\in{\tt C}(\mathbf 1,c), \varphi\in{\tt C}(c,\mathbf 1^*).$ From $\varphi\circ\Phi\in{\tt C}(\mathbf 1,\mathbf 1^*)$ we  obtain positivity of the map $\mathcal S(\varphi\circ\Phi)=\mathcal S(\varphi)\circ\mathcal S(\Phi)$ on $A(\mathbf 1)\supseteq B(\mathbf 1)\ni 1$, which implies $\mathcal S(\Phi)(1)\in A(c).$ This proves $B(c)\subseteq A(c)$. 
\ref{co3} is obvious. 
For \ref{co5} note that under the identification ${\tt C}(c^*, c^*)\cong{\tt C}(\mathbf 1, c\ot c^*)\cong{\tt C}(c^*\otimes c,\mathbf 1^*)$ and after  applying $\mathcal S$,  the identity corresponds to the maps $1\mapsto m_{\mathcal S(c)}$ and $\textrm{ev}_{\mc{S}(c)}$, respectively.
\end{proof}

\subsection{General conic systems}\label{ssec:acs}

Having fixed a stem $\mc{S}$, we now define a general conic system for the fixed stem $\mc{S}$ on a given vector space $X\in{\tt FVec}$, and call this an $\mc{S}$-system on $X$. This generalizes the notion of an abstract operator system, as shown in \cref{tab:aos_to_gcs}. 
Note that $X\in{\tt FVec}$ is in general not an element of $\mc{S}({\tt C}),$ just like for operator systems the vector space $X$ need not be a space of Hermitian matrices itself.
The following definition will be stated in less categorical terms in \cref{rem:acs} below and is visualized in (\cref{fig:ssystem}). 

\begin{figure}[t]
    \centering
    \includegraphics{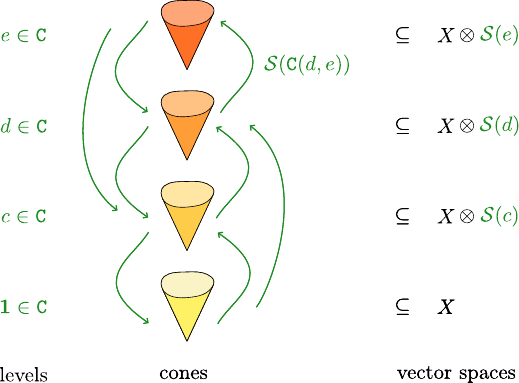}
    \caption{\small Fixing a stem (green) corresponds to fixing the vector spaces of the levels ($\mc{S}(c)$ for $c \in {\tt C}$) and the morphisms between them ($\mc{S}({\tt C}(d,e))$ for $d,e \in {\tt C}$). An $\mc{S}$-system $\mc{G}$ on $X$ is given by a convex cone inside $X \otimes \mc{S}(c)$ for every $c \in {\tt C}$, here visualized by the cones, where colors correspond to the different levels. The cones should be compatible with respect to the morphisms, i.e. if $x \in \mc{G}$ then for all morphisms in the stem, its image is in $\mc{G}$ (compare with \cref{fig:aos_def}). }
    \label{fig:ssystem}
\end{figure}

 \begin{definition}[$\mc{S}$-systems and completely positive maps]\label{def:acs} Let  $\mathcal S\colon{\tt C}\to{\tt FVec}$ be a stem and let $X\in{\tt FVec}$. 
\begin{enumerate}[label=$(\roman*)$]
\item An \emph{${\mathcal S}$-system on $X$} is a functor $\mathcal G\colon {\tt C}\to {\tt FCone}$ that extends the innocent bystander/tensor functor $X\otimes\cdot\colon {\tt FVec}\to{\tt FVec}$ (\cref{ex:saf}\ref{saf3}) with respect to\ $\mathcal S$ and the forgetful functor $\mathcal F$ (\cref{ex:saf}\ref{saf1}), i.e.\ for which the following diagram commutes:
$$
\xymatrix{{\tt FVec}\ar[r]^{X\otimes\cdot } & {\tt FVec} \\ {\tt C} \ar[r]^{\mathcal G}\ar[u]_{\mathcal S}& {\tt FCone}\ar[u]^{\mathcal F}}
$$ 
We call $\mathcal G(c)$ the \emph{cone at level $c$} of $\mathcal G$.

\item\label{acs2} If $\mathcal G, \mathcal E$ are  ${\mathcal S}$-systems on $X$ and $Y$ respectively, a linear map $\Psi\colon X\to Y$ is called \emph{completely positive with respect to $\mathcal G,\mathcal E$} if the family $$(\Psi\otimes{\rm id}_{\mathcal S(c)})_{c\in{\tt C}}$$ is a natural transformation from $\mathcal G$ to $\mathcal E$.\qedhere
\end{enumerate}
\end{definition}

Note that the functor $\mathcal G$ does not need to be star autonomous or even (star) monoidal.

\begin{remark}\label{rem:acs}
We can also state \cref{def:acs} more explicitly. An  ${\mathcal S}$-system $\mathcal G$ on $X$ consists of a closed convex cone $$\mathcal G(c)\subseteq X\otimes \mathcal S(c)$$ for every $c\in{\tt C}$. Furthermore, for every  $\Phi\in{\tt C}(c,d)$ we must have $${\rm id}_X\otimes \mathcal S(\Phi)\in\mathrm{Pos}(\mathcal G(c),\mathcal G(d))$$ i.e. the collection of all $\mathcal G(c)$ must be closed under applications of morphisms coming from {\tt C} via $\mc{S}$.
The map $\Psi\colon X\to Y$ is completely positive with respect to  $\mathcal G,\mathcal E$ if $$\Psi\otimes{\rm id}_{\mathcal S(c)}\in\mathrm{Pos}(\mathcal G(c),\mathcal E(c))$$ for all $c\in{\tt C}$.
\end{remark}

 We will prove some results before illustrating the concepts in \cref{ex:acs} below. For the upcoming results one can keep the original abstract operator systems in mind as an example. 
 
 The following lemma will be used to prove existence of a smallest and largest $\mc{S}$-system over a given cone $\mc{G}(\mathbf{1})$ at the base level (\cref{thm:minmax}). For operator systems, the smallest and largest systems arise from minimal and maximal tensor products with the psd cone. We will see that we need the $B$ cones for the smallest, and the $A$ cones for the largest systems. 
 
\begin{lemma}[]\label{lem:ext}
Let $\mc{G}$ be an  ${\mathcal S}$-system on $X$. For all $c,d\in{\tt C}$ we have
\begin{enumerate}[label=$(\roman*)$]
\item\label{sys1} $\mc{G}(c)\omin B(d)\subseteq \mc{G}(c\otimes d)$ and $\mc{G}(c\ot d)\subseteq \mc{G}(c)\omax A(d)$.
\item\label{sys2} $\mathcal {\rm int}(\mc{G}(c))\omin\mathcal {\rm int}(B(d))\subseteq \mathcal {\rm int}(\mc{G}(c\otimes d))$. 
\end{enumerate}
\end{lemma}
\begin{proof}
For \ref{sys1} we take  $d\in\mc{G}(c)$ and $a=\mathcal S(\Phi)(1)\in B(d)$ for some $\Phi\in{\tt C}(\mathbf 1, d).$ Then  $$d\otimes a=( {\rm id}_X\otimes (\underbrace{{\rm id}_{\mathcal S(c)}\otimes \mathcal S(\Phi)}_{=\mathcal S({\rm id}_{c}\otimes\Phi)})(d\otimes 1)\in\mc{G}(c\otimes d).$$ 
Similarly, for $\varphi\in{\tt C}(d,\mathbf 1^*)$ we have $({\rm id}_{c^*}\otimes\varphi^*)^*\in{\tt C}(c\ot d,c)$ and for $x\in\mc{G}(c\ot d)$ thus $$({\rm id}_{X\otimes \mathcal S(c)}\otimes\mathcal S(\varphi))(x)=({\rm id}_X\otimes\mathcal S (({\rm id}_{c^*}\otimes\varphi^*)^*))(x)\in\mc{G}(c).$$ This proves $x\in\mc{G}(c)\omax A(d).$

For \ref{sys2}  take $a_1\in{\rm int}(\mc{G}(c)),a_2\in{\rm int}(B(d))$ and  $x=\sum_i v_i\otimes w_i\in  (X\otimes \mathcal  S(c))\otimes \mathcal S(d)$ arbitrary. Choose $\lambda>0$ with $$\lambda a_1\pm v_i \in \mc{G}(c)$$ for all $i$, using that $a_1$ is an interior point of $\mc{G}(c)$. Furthermore, write $w_i=b_i- \bar{b}_i \mbox{ with }b_i,\bar{b}_i\in B(d)$, using that $B(d)$ has nonempty interior. Then we have $$\lambda a_1\otimes \sum_i(b_i+\bar{b}_i)+ x =\sum_i (\lambda a_1+v_i)\otimes b_i + (\lambda a_1-v_i)\otimes \bar{b}_i\in \mc{G}(c)\omin B(d).$$ Now choose $\gamma>0$ with $\gamma a_2- \sum_i (b_i+\bar{b}_i) \in B(d),$ obtain $ a_1\otimes \gamma a_2- a_1\otimes \sum_i (b_i+\bar{b}_i) \in \mc{G}(c)\omin B(d),$ and thus  $$\lambda\gamma(a_1\otimes a_2) + x \in \mc{G}(c)\omin B(d)\subseteq\mc{G}(c\otimes d),$$ using \ref{sys1}. This proves that $a_1\otimes a_2$ is an interior point of $\mc{G}(c\otimes d)$.
\end{proof}

\cref{lem:ext}\ref{sys2} shows in particular that if $\mc{G}(\mathbf{1})$ has non-empty interior, as well as all $B(d)$ for all $d$, then all cones $\mc{G}(d)$ have nonempty interior. This generalizes the result in \cite{Fritz17} that an order unit at level $1$ in an aos induces an order unit at every level, since the psd cone has nonempty interior in every dimension. 

\begin{figure}[t]
    \centering
    \includegraphics{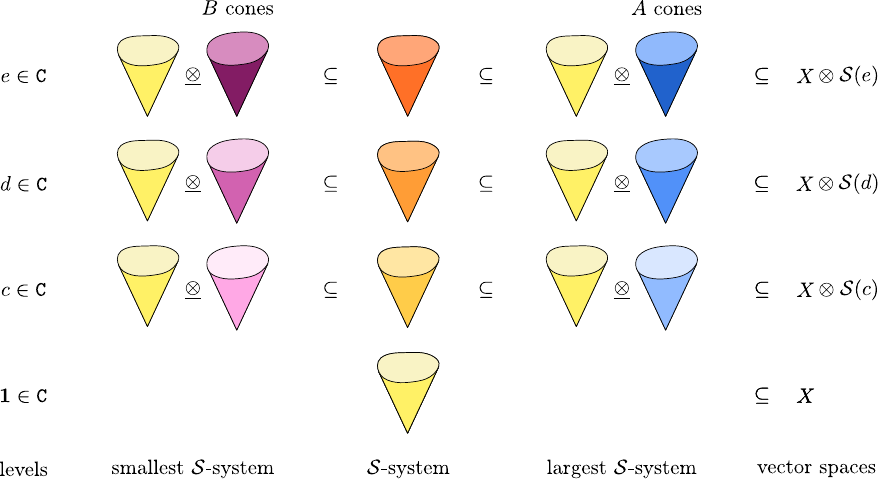}
    \caption{\small Every cone in $X$ (yellow, at level $\mathbf{1}$) can be extended to an $\mc{S}$-system by choosing the other cones in a compatible way (orange colors). The compatibility with the morphisms between the levels results in a smallest and largest extension of the yellow cone, given by the minimal tensor product with the $B$ cones and the maximal tensor product with the $A$ cones, respectively at every level (\cref{thm:minmax}). }
    \label{fig:acs}
\end{figure}

\begin{proposition}[Smallest and largest system (\cref{fig:acs})]\label{thm:minmax} Let $\mathcal S$ be a stem.
\begin{enumerate}[label=$(\roman*)$]
    \item\label{symin} For $(D,X)\in{\tt FCone}$ and $c\in{\tt C}$ define
 $${\underline{\mc{G}}}_D(c):=D\omin B(c).$$  Then ${\underline{\mc{G}}}_D$ is the smallest $\mc{S}$-system on $X$ that contains the cone $D$ at level $\mathbf 1$.  
\item\label{symax}  For $(D,X)\in{\tt FCone}$ and $c\in{\tt C}$ define $$\overline{\mc{G}}_D(c):=D\omax A(c).$$ Then $\overline{\mc{G}}_D$ is the largest $\mc{S}$-system on $X$ that is contained in $D$ at level $\mathbf{1}^*$.
\end{enumerate}
\end{proposition}
\begin{proof}
($i$) Closedness of the minimal tensor product was proven in \cite{debruyn}, so $\underline{\mc{G}}_D$ is an  ${\mathcal S}$-system on  $X$. \cref{lem:ext}\ref{sys1} implies that whenever $\mc{G}$ is another  ${\mathcal S}$-system on $X$, with $D\subseteq \mc{G}(\mathbf 1)$, then  $\underline{\mc{G}}_D(c)\subseteq \mc{G}(c)$ for all $c\in{\tt C}$. So $\underline{\mc{G}}_D$ is the smallest  ${\mathcal S}$-system containing $D$ at level $\mathbf 1$. Note that if $B(\mathbf 1)=\mathbb R_{\geqslant 0},$ then $\underline{\mc{G}}_D(\mathbf 1)=D\omin\mathbb R_{\geqslant 0}=D.$ However, if $B(\mathbf 1)=\mathbb R$  then $$\underline{\mc{G}}_D(\mathbf 1)=D\omin\mathbb R={\rm span}(D)$$ can be strictly larger than $D$. 

($ii$) By closedness of the maximal tensor product, $\overline{\mc{G}}_D$ is an  ${\mathcal S}$-system on $X$. Now \cref{lem:ext}\ref{sys1} implies that whenever $\mc{G}$ is another ${\mathcal S}$-system with $\mc{G}(\mathbf 1^*)\subseteq D$, then $\mc{G}(c)\subseteq   \overline{\mc{G}}_D(c)$ for all $c\in{\tt C}$. So $\overline{\mc{G}}_D$ is the largest  ${\mathcal S}$-system contained in $D$ at level $\mathbf 1^*$. If $A(\mathbf 1^*)=\mathbb R_{\geqslant 0}^\vee,$ then $\overline{\mc{G}}_D(\mathbf 1^*)=D,$ but if $A(\mathbf 1^*)=\{0\}$, then $$\overline{\mc{G}}_D(\mathbf 1^*)=D\omax \{0\}={\rm span}(D^\vee)^{\vee}=D\cap -D$$ can be strictly smaller than $D$. 
\end{proof}

\begin{figure}[t]
    \centering
    \includegraphics{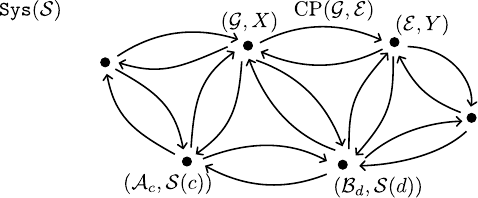}
    \caption{\small The category of $\mc{S}$-systems ${\tt Sys}(\mc{S})$ has objects of the form $(\mc{G},X)$ where $\mc{G}$ is an  ${\mathcal S}$-system on $X$ and $X\in{\tt FVec}$, and morphisms are completely positive maps. $\mc{A}_c$ and $\mc{B}_d$ are the intrinsic and cointrinsic systems on $\mc{S}(c)$ and $\mc{S}(d)$, respectively (\cref{def:coint_sys}). }
    \label{fig:as_cat}
\end{figure}

For any fixed stem $\mathcal S$, there are many $\mc{S}$-systems. First, one can choose any $X$ and consider $\mc{S}$-systems over $X$, and secondly there are many options of $\mc{S}$-systems $\mc{G}$ over any given $X$. The collection of all these ${\mathcal S}$-systems for a fixed stem forms itself a category denoted {\tt Sys($\mathcal S$)} (\cref{fig:as_cat}). Its objects are pairs $(\mc{G},X)$ where $X\in{\tt FVec}$ and $\mc{G}$ is an  ${\mathcal S}$-system on $X$, often denoted just by $\mc{G}$. The completely positive maps (in the sense of \cref{def:acs}\ref{acs2}) form the morphisms of the category {\tt Sys($\mathcal S$)}, and we denote the morphism set from $(\mc{G},X)$ to $(\mc{E},Y)$ by $\mathrm{CP}(\mc{G},\mathcal E).$ Note that $$\mathrm{CP}(\mc{G},\mathcal E)\subseteq \mathrm{Pos}(\mc{G}(\mathbf 1),\mathcal E(\mathbf 1))\subseteq \mathrm{Lin}(X,Y)$$ is itself a closed convex cone.

\begin{figure}[t]
    \centering
    \includegraphics{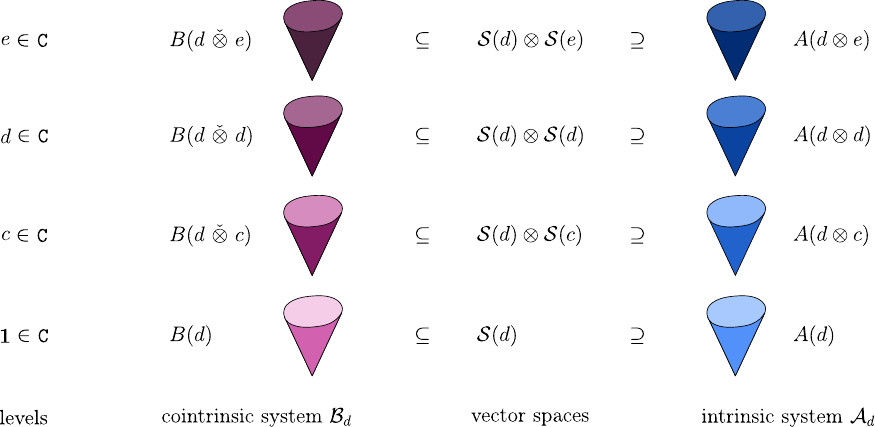}
    \caption{\small For every stem $\mc{S}$ and every $d \in {\tt C}$, there are two special $\mc{S}$-systems on $\mc{S}(d)$, the intrinsic system $\mc{A}_d$ and the cointrinsic system $\mc{B}_d$. The intrinsic system consists of the cone $A(d)$ at the base level, and the cone $A(d \otimes c)$ for any level $c \in {\tt C}$. The cointrinsic system consists of the cone $B(d)$ at the base level and the cone $B(d~\check{\otimes}~c)$ for any level $c \in {\tt C}$.}
    \label{fig:intrinsic}
\end{figure}

If we choose $X$ to be a vector space in the image of the functor $\mathcal S$, i.e. $X=\mc{S}(d)$ for some $d \in {\tt C}$, the higher level spaces are of the form $\mc{S}(d) \otimes \mc{S}(c)=\mathcal S(d\otimes c)$ for all $c \in {\tt C}$. In this case, there are two special and important $\mc{S}$-systems, defined by the intrinsic and cointrinsic cones inside the level spaces (\cref{fig:intrinsic}). 
These are the analogue of aos on the space $\textrm{Her}_d$ with the psd cone at every level. 
\begin{definition}[(Co)Intrinsic systems]\label{def:coint_sys} Let $\mathcal A_{d},\mathcal B_{d}$ be the $\mc{S}$-systems on $\mathcal S(d)$ defined by 
\begin{align*} \mathcal A_{d}(c)&:=A(d\otimes c)\subseteq\mathcal S(d\otimes c)=\mathcal S(d)\otimes\mathcal S(c), \\\mathcal  B_{d}(c)&:=B(d\ot c)\subseteq\mathcal S(d\ot c)=\mathcal S(d)\otimes\mathcal S(c). \end{align*} 
We call $\mathcal A_{d}$ the \emph{intrinsic system} and $\mathcal B_{d}$ the \emph{cointrinsic system} of $\mathcal S$ on $\mc{S}(d)$. 
\end{definition}

\begin{remark}\label{rem:cp}For $\Phi\in{\tt C}(c,d)$ the map $\mathcal S(\Phi)$ is completely positive with respect to $\mathcal A_{c},\mathcal A_{d}$ and also with respect to $\mathcal B_{c},\mathcal B_{d}$ (the converse follows from \cref{cor:chkr}). 
\end{remark}
The intrinsic and cointrinsic system receive this name because they are intrinsic to the stem $\mc{S}$. They play an important role in the rest of the paper, first of all to define (concrete) realizations of $\mc{S}$-systems (\cref{def:real}), just like the psd operators play that role in the original Choi--Effros theorem. Before we come to that definition, let us explain several useful constructions in {\tt Sys}($\mathcal S$).

\begin{definition}[Operations on $\mc{S}$-systems]\label{def:operations}\quad
\begin{enumerate}[label=$(\roman*)$]
\item\label{sy2} If $\mc{G},\mathcal E$ are  $\mathcal S$-systems on the same space $X,$ we say that $\mathcal E$ \emph{contains} $\mc{G}$, and write $\mc{G}\subseteq \mathcal E$, if  ${\rm id}_X\in\mathrm{CP}(\mc{G},\mathcal E).$ This means that $\mc{G}(c)\subseteq\mathcal E(c)$ for all $c\in{\tt C}.$\footnote{Note that we already used this notion when introducing the smallest and largest system over a cone.}

\item\label{sy3} For some given index set $I$ and $i\in I$, let $\mc{G}_i$ be   $\mathcal S$-systems on the same space $X.$ The  \emph{sum}  $\sum_{i\in I} \mc{G}_i$ is the smallest  ${\mathcal S}$-system on $X$ such that $${\rm id}_X\in\mathrm{CP}\left(\mc{G}_i,\sum_{i\in I} \mc{G}_i\right)$$ holds for all $i\in I$.  Explicitly, $(\sum_{i\in I}\mc{G}_i)(c)$ is the closed convex cone generated by all $\mc{G}_i(c)$ together. Similarly, the \emph{intersection} $\bigcap_{i\in I}\mc{G}_i$ is the largest  ${\mathcal S}$-system on $X$ such that $${\rm id}_X\in\mathrm{CP}\left(\bigcap_{i\in I}\mc{G}_i, \mc{G}_i\right)$$ holds for all $i\in I$. Explicitly we have $(\bigcap_{i\in I}\mc{G}_i)(c)=\bigcap_{i\in I}\mc{G}_i(c)$ for all $c\in{\tt C}.$

\item\label{sy4} If $\mc{G},\mathcal E$ are  $\mathcal S$-systems on $X$ and $Y$ respectively, their \emph{direct sum} is an  ${\mathcal S}$-system on $X\oplus Y$: $$(\mc{G}\oplus\mathcal E)(c) :=\mc{G}(c)\oplus\mathcal E(c)\subseteq (X\otimes \mathcal S(c))\oplus (Y\otimes \mathcal S(c))= (X\oplus Y)\otimes \mathcal S(c).$$

\item\label{sy5} If $\Psi\in\mathrm{Lin}(X,Y)$ and $\mathcal E$ is an  $\mathcal S$-system on $Y$, its \emph{inverse image} $\Psi^{-1}(\mathcal E)$ is the largest  ${\mathcal S}$-system on $X$ that makes  $\Psi$ completely positive. Explicitly, we have $$(\Psi^{-1}(\mathcal E))(c)=(\Psi\otimes{\rm id}_{\mathcal S(c)})^{-1}(\mathcal E(c)).$$

\item\label{sy6} If $\Psi\in\mathrm{Lin}(X,Y)$ and $\mc{G}$ is an  $\mathcal S$-system on $X$, its \emph{image} $\Psi(\mc{G})$ is the smallest  ${\mathcal S}$-system on $Y$ that makes $\Psi$ completely positive. Explicitly, we have $$(\Psi(\mc{G}))(c)=\overline{(\Psi\otimes{\rm id}_{\mathcal S(c)})(\mc{G}(c))},$$ where the bar denotes the closure. \qedhere
\end{enumerate}
\end{definition}

\begin{figure}[t]
    \centering
    \includegraphics{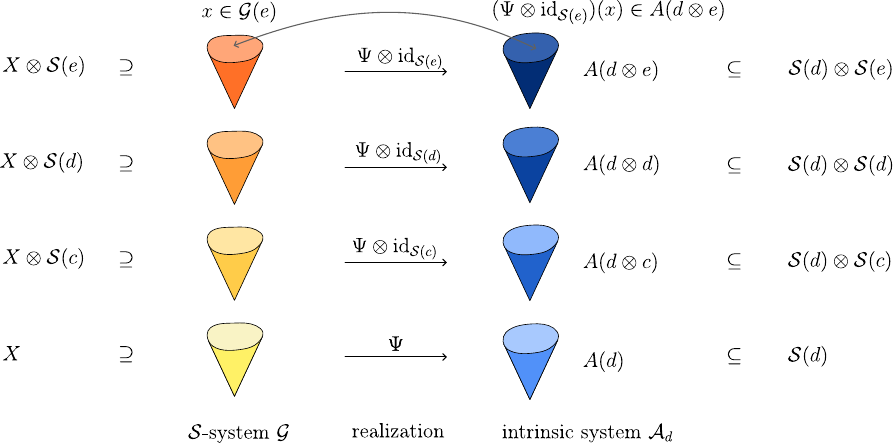}
    \caption{\small This $\mc{S}$-system $\mc{G}$ has a concrete realization on $\mc{A}_d$, meaning there is a linear map $\Psi: X \to \mc{S}(d)$ such that for every $e\in {\tt C}$, $\mc{G}(e) = (\Psi \otimes \textrm{id}_{\mc{S}(e)})^{-1}A(d \otimes e)$.}
    \label{fig:realization}
\end{figure}

We are now ready to introduce $\mc{S}$-systems with a concrete realization, generalizing the notion of a concrete operator system. In this definition, the `symmetry' between the cointrinsic and intrinsic systems will be broken for the first time. 

\begin{definition}[Concrete realizations, $\mc{S}$-hedra and finitely generated $\mc{S}$-systems]\label{def:real} \quad
\begin{enumerate}[label=$(\roman*)$]
\item An $\mc{S}$-system $\mc{G}$ on $X$ \emph{has a concrete realization} (on $\mathcal A_{d}$) if it is the inverse image  of the  intrinsic system $\mathcal A_d$ under a linear map $X \to \mc{S}(d)$ (\cref{fig:realization}).

\item An $\mc{S}$-system $\mc{G}$ is an \emph{$\mathcal S$-hedron} if it is a finite intersection of systems with a concrete realization. The image of an $\mathcal S$-hedron under a linear map is an \emph{$\mathcal S$-hedrop}\footnote{For aos, an $\mathcal S$-hedrop is a (free) spectrahedrop, which is the terminology for a projection of a spectrahedron (or a spectrahedron that has been \emph{drop}ped to the floor) \cite{Helton16}.}.

\item An  ${\mathcal S}$-system $\mc{G}$ is \emph{finitely generated} if there are $a_i\in\mc{G}(c_i)$ for $i=1,\ldots, n$, for which $\mc{G}$ is the smallest system containing $a_1,\ldots, a_n$. In this case we also denote $\mc{G}$ by $\mathcal W(a_1,\ldots, a_n)$\footnote{This notation originates from the German word `Wertebereich' for numerical range, which is usually the set of numbers obtained by applying all vector contractions to a matrix. Here we apply all morphisms coming from the stem to the generators, and in this sense obtain a generalized noncommutative multivariate numerical range.}. \qedhere
\end{enumerate}
\end{definition}

\begin{example}[Examples of general conic systems]\label{ex:acs} Let us illustrate the above constructions for our running examples, the stems of \cref{ex:indstr}.
\begin{enumerate}[label=$(\roman*)$]
\item\label{vsys}\textbf{(Vector systems)} Consider the vector stem $\mathcal S={\rm id}\colon{\tt FVec}\to{\tt FVec}$ from \cref{ex:indstr}\ref{FVec}. We call the arising $\mathcal S$-systems just \emph{vector systems}. Let $\mc{G}$ be a vector system on the space $X$. Since each $\mc{G}(V)\subseteq X\otimes V$ is closed under the maps ${\rm id}_X\otimes \Phi$ for all $\Phi\in\mathrm{Lin}(V,V),$ each $\mc{G}(V)$ must be a subspace. It turns out that  $\mc{G}$ is uniquely determined by the subspace $U:=\mc{G}(\mathbb R)\subseteq X$ at the base level. In fact, since  $$U\omin B(V)=U\omin V=U\omax\{0\}=U\omax A(V)$$ for all $V\in{\tt FVec},$ this defines the only vector system over $U.$ 

For two vector systems $(\mc{G},X)$ and $(\mathcal E,Y),$
$\mathrm{CP}(\mc{G},\mathcal E)$ consists of those linear maps from $X$ to $Y$ that map $\mc{G}(\mathbb R)$ into $\mathcal E(\mathbb R).$ The intrinsic system $\mathcal A_V$ corresponds to the subspace $\{0\}$ and the cointrinsic system $\mathcal B_V$ corresponds to the full space $V$. Each vector system $\mc{G}$ on  $X$ has a concrete realization, given by any linear map $\Psi \in \mathrm{Lin}(X,V)$ whose kernel is exactly $\mc{G}(\mathbb R)$.
Every vector system is also finitely generated, for example  by a basis of the subspace $\mc{G}(\mathbb R)$. 

\item\label{ssys} \textbf{(Simplex systems)} Consider the simplex stem $\mathcal S\colon {\tt SCone}\to{\tt FVec}$ from \cref{ex:indstr}\ref{SCone}. We call the arising systems \emph{simplex systems.} 
Since $C \omin S = C \omax S$ for any convex cone $C$ and any simplex cone $S$, a simplex system $\mc{G}$ on the space $X$ is completely determined by its cone  $\mc{G}(\mathbb R_{\geqslant 0}) \subseteq X$ at the base level.  
In particular, every positive map is completely positive: $$\mathrm{CP}(\mc{G},\mc{E})=\mathrm{Pos}(\mc{G}(\mathbb R_{\geqslant 0}),\mc{E}(\mathbb R_{\geqslant 0})).$$ 
Intrinsic and cointrinsic systems coincide, and are given by
$$\mathcal B_D(S)=\mc{A}_D(S) = D \omin S=D \omax S.$$ 

A simplex system $\mc{G}$ has a concrete realization if and only if  $\mc{G}(\mathbb R_{\geqslant 0})$ is  the inverse image of a simplex cone, which just means that  $\mc{G}(\mathbb R_{\geqslant 0})$ is polyhedral. Finite intersections of such are still polyhedra, so having a concrete realization is equivalent to being a simplex-hedron. This is easily seen to be equivalent to $\mc{G}$ being finitely generated.

\item\label{os}\textbf{(Operator systems)} The operator stem  $\mathcal S\colon{\tt FHilb}^{\rm op}\to{\tt FVec}$ from \cref{ex:indstr}\ref{aos} gives rise to the well-known notion of operator systems. Indeed, a system $\mc{G}$ on $X$ consists of a cone $\mc{G}(H)\subseteq X\otimes\mathsf B(H)_{\rm her}$ for each finite-dimensional Hilbert space $H$, such that the family is closed under compression maps on the second tensor factor. The notion of completely positive maps coincide with the usual one used in this context. Since the operator stem is self-dual and compact closed, intrinsic and cointrinsic systems coincide and equal the system of positive semidefinite operators: $$\mathcal B_{\tilde H}(H)=\mathcal A_{\tilde H}(H)={\rm Psd}(\tilde H\otimes H).$$ After fixing a basis of $X$, a system with a concrete realization is given by $$\mc{G}(H)=\{A\in\mathsf B(H)_{\rm her}^d\mid M_1\otimes A_1+\cdots +M_d\otimes A_d \in \textrm{Psd}(\tilde H \otimes H)\}$$ for certain $M_1,\ldots, M_d\in\mathsf B(\tilde H)_{\rm her}.$
So systems with a concrete realization coincide with $\mathcal S$-hedra, and are also known as \emph{free spectrahedra} or \emph{operator systems with a finite-dimensional realization}.\qedhere
\end{enumerate}
\end{example}

\subsection{Duality}\label{ssec:duality} Just like convex cones have dual cones, every $\mc{S}$-system has a dual system. If $\mc{G}$ is an $\mc{S}$-system on $X$, its dual system is an $\mc{S}$-system on $X',$ consisting of the level-wise duals of the cones of $\mc{G}$. There is also another description of the dual in terms of completely positive maps to the intrinsic system, which is sometimes very useful. Duality of $\mathcal S$-systems is a very helpful tool; for example, we will see in \cref{thm:dual} that being realizable is dual to being finitely generated, and the maximal system is dual to the minimal. Some properties can be easily proven via the dual property of the dual system. 

Elements of the dual cone are elements of the dual vector space. Before defining the dual of an $\mc{S}$-system $\mc{G}$ in \cref{def:dual}, we use the identification between vector spaces, tensors and linear maps \eqref{eq:isom} to connect the level-wise dual cones of an $\mc{S}$-system $\mc{G}$ to the cone of completely positive maps into an intrinsic system. Note that here the symmetry between the intrinsic and cointrinsic system is broken again, because only the intrinsic system is connected to the dual cones. 

\begin{proposition}[Level-wise dual]\label{prop:dual}Let  $(\mc{G}, X)$ be an $\mathcal S$-system and $c\in{\tt C}$. Using the identification $(X\otimes\mathcal S(c^*))'\cong X'\otimes \mathcal S(c)\cong{\rm Lin}(X,\mathcal S(c)),$ we obtain $$\mc{G}(c^*)^\vee ={\rm CP}(\mc{G},\mathcal A_c).$$
\end{proposition}
\begin{proof}
First take $a\in\mc{G}(c^*)^\vee$ and $b\in\mc{G}(d)$. Then for any $\varphi\in {\tt C}(c\otimes d,\mathbf 1^*)\cong {\tt C}(d,c^*)$ we have  $$\mathcal S(\varphi)(a\otimes{\rm id}_{\mathcal S(d)})(b)=a((\underbrace{{\rm id}_X\otimes\mathcal S(\varphi))(b)}_{\in\mc{G}(c^*)})\geqslant 0.$$ This proves $(a\otimes{\rm id}_{\mathcal S(d)})(b)\in A(c\otimes d)$ and thus $a\in\mathrm{CP}(\mc{G},\mathcal A_c).$
Conversely, take $a\in \mathrm{CP}(\mc{G},\mathcal A_c).$ Then for $b\in\mc{G}(c^*)$ we have $$(a\otimes {\rm id}_{\mathcal S(c^*)})(b)\in A(c\otimes  c^*)$$ and thus $$0\leqslant\textrm{ev}_{\mc{S}(c^*)}((a\otimes{\rm id}_{\mathcal S(c^*)})(b))=a(b)$$ by \cref{lem:cones}\ref{co5}. This proves $a\in\mc{G}(c^*)^\vee.$
\end{proof}    

\begin{definition}[Dual systems]\label{def:dual}
Let $(\mc{G},X)\in{\tt Sys}(\mathcal S)$.  For $c\in{\tt C}$ we define $$\mc{G}^\vee(c):=\mc{G}(c^*)^\vee=\mathrm{CP}(\mc{G},\mathcal A_c)$$ and call  $\mathcal{G}^\vee$ the \emph{dual system}  of $\mc{G}.$ 
\end{definition}

The following theorem collects several results on dual systems (\cref{fig:duality}). The most important statements are \ref{du0} where we prove that $\mc{G}^\vee$ is an $\mathcal S$-system, \ref{duminmax} showing that minimal and maximal $\mc{S}$-systems are dual to each other, and \ref{du6} showing that finitely generated $\mc{S}$-systems are dual to $\mc{S}$-hedra. 

\begin{figure}[th]
    \centering
    \includegraphics{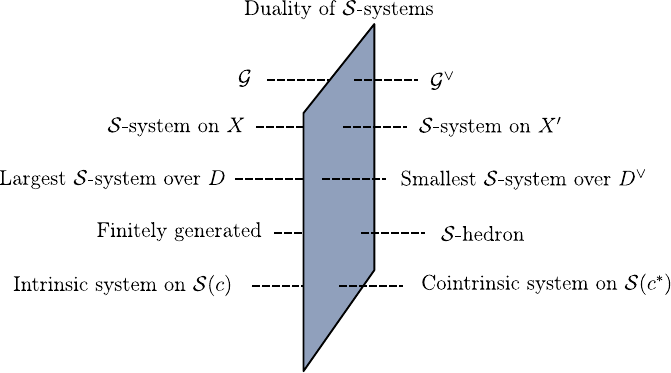}
    \caption{\small Some of the properties of dual $\mc{S}$-systems proven in \cref{thm:dual} and \cref{thm:selfdual} are summarized in this `duality window'.}
    \label{fig:duality}
\end{figure}

\begin{theorem}[Properties of dual systems]\label{thm:dual} Given a stem $\mc{S}$ and vector spaces $X,Y$, let $(\mc{G}, X), (\mathcal E, X),(\mathcal D,Y)\in{\tt Sys(\mathcal S)},$  $\Psi\in{\rm Lin}(X,Y)$. Then the following holds: 
\begin{enumerate}[label=$(\roman*)$]
\item\label{du0} $\mc{G}^\vee$ is an $\mathcal S$-system on $X'$.
\item\label{du1} Under the canonical identification $X= X''$ we have $(\mc{G}^{\vee})^{\vee}=\mc{G}.$
\item\label{du2} $\mc{G}\subseteq \mathcal E\Leftrightarrow \mathcal E^{\vee}\subseteq \mc{G}^{\vee}.$
\item\label{duminmax} $\underline{\mc{G}}_D^{\vee}= \overline{\mc{G}}_{D^\vee}$ and $\overline{\mc{G}}_D^{\vee}=\underline{\mc{G}}_{D^\vee}.$
\item\label{du3}  $(\mc{G}+ \mathcal E)^{\vee}=\mc{G}^{\vee}\cap \mathcal E^{\vee}$ and $(\mc{G}\cap \mathcal E)^{\vee}=\mc{G}^{\vee}+\mathcal E^{\vee}.$
\item\label{du4}  $\Psi(\mc{G})^{\vee}= (\Psi')^{-1}(\mc{G}^{\vee})$  and $\Psi^{-1}(\mathcal D)^{\vee}= \Psi'(\mathcal D^{\vee}).$
\item\label{du5} $\Psi\in{\rm CP}(\mathcal E,\mathcal D)$ if and only if $\Psi'\in{\rm CP}(\mathcal D^{\vee},\mathcal E^{\vee}).$
\item\label{du6} $\mc{G}$ is  finitely generated  $\Leftrightarrow$ $\mc{G}^{\vee}$ is an $\mathcal S$-hedron (and vice versa). 

\item\label{du7}$\mc{G}$ is the inverse image of some  finitely generated system $\Leftrightarrow$ $\mc{G}^{\vee}$ is an $\mathcal S$-hedrop (and vice versa). 
\end{enumerate}
\end{theorem}
\begin{proof}
For \ref{du0} we take $f \in {\rm CP}(\mc{G}, \mc{A}_c) \subseteq X' \otimes \mc{S}(c)$, so for all $d \in {\tt C}$ we have $$(f \otimes \textrm{id}_{\mathcal S(d)})\left(\mc{G}(d)\right) \subseteq A(c \otimes d).$$ 
Therefore, for every $\Phi \in {\tt C}(c,e)$ and every $e \in {\tt C}$, $$
(\mc{S}(\Phi)\circ f \otimes \textrm{id}_{\mc{S}(d)})\left( \mc{G}(d)\right)\subseteq (\mc{S}(\Phi) \otimes \textrm{id}_{\mc{S}(d)})\left( A(c \otimes d)\right) \subseteq A(e \otimes d),$$ by \cref{rem:cp}.  So $\mathcal S(\Phi)\circ f\in {\rm CP}(\mathcal G,\mathcal A_e),$ which was to be shown.
\ref{du1} is clear from the classical biduality theorem, \ref{du2} follows directly from \ref{du1}, and \ref{duminmax} and \ref{du3} can be verified easily. 
\ref{du4} is a direct  computation,  \ref{du5} follows from \ref{du4}. 
For \ref{du6} we use the definition of the dual in terms of completely positive maps. First observe that if $\mc{G}$ is generated by $a_i\in\mc{G}(c_i)$ for $i=1,\ldots, n$, then $b\in\mc{G}^{\vee}(c)=\mathrm{CP}(\mc{G},\mathcal A_{c})$ if and only if $$(b\otimes{\rm id}_{\mathcal S(c_i)})(a_i)\in A(c\otimes c_i) \quad\mbox{ for } i=1,\ldots, n.$$ But again, $(b\otimes{\rm id}_{\mathcal S(c_i)})(a_i)\in A(c\otimes c_i) $ is equivalent to $(a_i\otimes{\rm id}_{\mathcal S(c)})(b)\in A(c_i\otimes c)$. This shows that $\mc{G}^{\vee}$ is the intersection of the  systems having a concrete realization on $\mathcal A_{c_i}$ via the maps $a_i\colon X'\to \mathcal S(c_i)$. The converse now follows from \ref{du1}.  \ref{du7} is clear from \ref{du6}, \ref{du4} and \ref{du1}.
\end{proof}

Intrinsic and cointrinsic systems are dual to each other, and the cointrinsic systems are generated by a single element: the maximally entangled state.

\begin{theorem}[Duals of intrinsic systems]\label{thm:selfdual} For $d\in{\tt C}$ we have $$\mathcal A_{d}^{\vee}=\mathcal B_{d^*}= \mathcal W(m_{\mathcal S(d^*)}),$$ where $\mathcal W(m_{\mathcal S(d^*)})$ is the $\mc{S}$-system generated by $m_{\mathcal S(d^*)}$ (\cref{def:real}). In addition, if $\mathcal S$ is compact closed and self-dual, then $\mathcal A_d^{\vee}=\mathcal A_{d^*}.$
\end{theorem}
\begin{proof}
That $\mathcal A_d^{\vee}=\mathcal B_{d^*}$ is obvious from \cref{lem:cones}\ref{co1}. Now by the proof of \cref{thm:dual}\ref{du6}, the dual of the system generated by $m_{\mathcal S(d^*)}$ is $\mathcal A_{d}$.  By biduality, $\mathcal A_{d}^{\vee}$ equals $\mathcal W(m_{\mathcal S(d^*)}).$ Finally, in the compact closed and self-dual case we have \begin{equation*}\mathcal A_d^\vee(c)=A(d\otimes c^*)^\vee=B(d^*\ot c)=A(d^*\ot c)=A(d^*\otimes c)=\mathcal A_{d^*}(c). \qedhere \end{equation*}
\end{proof}

\begin{example}[Examples of dual systems]\label{ex:dual} Let us illustrate dual systems for the $\mc{S}$-systems of \cref{ex:acs}.
    \begin{enumerate}[label=$(\roman*)$]
    \item\label{dual_ex_1} \textbf{(Dual vector systems)} The vector system $\mc{G}$ on $X$ is uniquely determined by the subspace $U=\mc{G}(\mathbb R)\subseteq X.$ From $$\mc{G}^\vee(V)=\mc{G}(V')^\vee=(U\omin V')^\vee=U^\vee \omin V$$ we see that the dual  corresponds to the dual/complement on the level of subspaces: $$U^\vee=U^\perp=\{f\in X'\mid f\equiv 0 \mbox{ on } U\}.$$ 
  
    \item\label{dual_ex_2} \textbf{(Dual simplex systems)} For a simplex system $\mc{G}$ on $X$, which is uniquely determined by its cone $\mc{G}(\mathbb R)$ at the base level, the dual corresponds to the classical dual at the base level.
    \item\label{dual_ex_3} \textbf{(Dual operator systems)} For an abstract operator system $\mc{G}$ over a vector space $\C^d$, the \emph{free} dual system is defined as \cite{Berger21}:
    $$\mc{G}^\vee(s) \coloneq \big\{ (A_1,\ldots, A_d) \in \mathrm{Her}_s(\C)^d \mid \sum_{i=1}^d B_i^\mathrm{T} \otimes A_i\  \mathrm{psd},\  \forall (B_1, \ldots, B_d) \in \mc{G}\big\},$$
where $(\cdot)^\textrm{T}$ denotes the transpose. This is exactly recovered by our \cref{def:dual}, albeit more conceptually in terms of completely positive maps. Moreover, it seems to have been unknown so far that the free dual coincides with the more obvious level-wise dual of the cones. 
     Duality of intrinsic systems implies that every positive semidefinite matrix in ${\rm Her}_m\otimes {\rm Her}_n$ is obtained from the maximally entangled state   $$\sum_{i,j=1}^m E_{ij}\otimes E_{ij}\in{\rm Her}_m\otimes {\rm Her}_m$$ by sums of compressions of the second tensor factor. \qedhere
    \end{enumerate}
\end{example}

For a (co)intrinsic system on $\mathcal S(c)$, the level $c^*$ plays a special role, as certain  inclusions   only need to be verified at level $c^*$. 

\begin{corollary}[Inclusion]\label{cor:levelmax}
Let  $c\in{\tt C}$  and  $\mc{G}$ be an  ${\mathcal S}$-system on $\mathcal S(c)$. Then the following holds:
\begin{enumerate}[label=$(\roman*)$]
\item\label{max1} If $\mc{G}(c^*)\subseteq \mathcal A_c(c^*)$, then $\mc{G}\subseteq \mathcal A_c$.
\item\label{max2} If $\mathcal B_{c}(c^*)\subseteq \mc{G}(c^*)$, then $\mathcal B_{c}\subseteq \mc{G}$.
\end{enumerate}
\end{corollary}
\begin{proof}
For \ref{max1} let $a\in\mc{G}(d)$ for some $d\in{\tt C}$. Then for all $\Phi\in {\tt C}(c\otimes d,\mathbf 1^*)={\tt C}(d, c^*)$ we have 
$$\mathcal S(\Phi)(a)={\rm ev}_{\mathcal S(c^*)}(\underbrace{({\rm id}_{\mathcal S(c)}\otimes\mathcal S(\Phi))(a)}_{\in\mc{G}(c^*)\subseteq \mathcal A_{c}(c^*)=A(c\otimes c^*)})\geqslant 0.$$ This implies $a\in A(c\otimes d)=\mathcal A_c(d).$
 Statement \ref{max2} is clear from \cref{thm:selfdual}, since $\mathcal B_{c}$ is generated by $m_{\mathcal S(c)}\in\mathcal B_{c}(c^*).$ 
\end{proof}

\subsection[Theorems about aos]{Generalizing theorems about abstract operator systems}\label{ssec:theorems}

Let us now revisit theorems on abstract operator systems and generalize them to general conic systems. In particular, we generalize the Choi--Kraus decomposition of completely positive maps (\cref{cor:chkr}), the existence of a concrete realization (\cref{cor:sep}), Arveson's extension theorem in finite dimensions (\cref{thm:ext}), Choi's theorem (\cref{thm:choi}) and a theorem on systems over polyhedral cones (\cref{thm:poly}). 
These now apply to many more examples of stems, and additionally, their proofs become short when phrased in this language. Even when applied to operator systems, we find the simplification of the proofs interesting.  

In the case of operator systems, by the Choi--Kraus decomposition every completely positive (cp) map can be written as a sum of matrix contractions. In terms of the framework introduced so far, this means that the cp maps between the (co)intrinsic systems (which coincide in the operator case) exactly arise from the morphisms in the underlying category. This statement is true for general $\mc{S}$-systems, as we now prove.
\begin{theorem}[Choi--Kraus decomposition of cp maps]\label{cor:chkr} For $c,d\in{\tt C}$ we have $${\rm CP}(\mathcal A_{c},\mathcal A_{d})={\rm CP}(\mathcal B_{c},\mathcal B_{d})=\overline{\rm cone}\ \mathcal S({\tt C}(c,d)).$$
\end{theorem}
\begin{proof} We have 
\begin{align*}\mathrm{CP}(\mathcal A_{c},\mathcal A_{d})&=\mathcal A_{c}^\vee(d)=A(c\otimes d^*)^\vee=B(c^*\ot d)\\&=\overline{\rm cone}\{\mathcal S(\Phi)(1)\mid\Phi\in{\tt C}(\mathbf 1,c^*\ot d)\} \\&= \overline{\rm cone}\ \mathcal S({\tt C}(c,d)).\end{align*}  The statement about $\mathrm{CP}(\mathcal B_{c},\mathcal B_{d})$ follows from duality.
\end{proof}

The Choi--Effros theorem states that every aos has a concrete realization. Its proof is largely a consequence of the Effros--Winkler separation theorem.
The generalization of these results is proven in the following combined theorem. 

\begin{theorem}[Separation and realization]\label{cor:sep}
Every  ${\mathcal S}$-system is an intersection of ${\mathcal S}$-hedra.
\end{theorem}
\begin{proof} For any ${\mathcal S}$-system $\mc{G}$, $\mc{G}^{\vee}$ is clearly the (potentially infinite) sum of finitely generated subsystems $$\mc{G}^{\vee} = \sum_i\mc{E}_i.$$ By \cref{thm:dual}\ref{du6} $\mc{E}_i^\vee$ is an $\mc{S}$-hedron for every $i$, and by \cref{thm:dual}\ref{du3} \begin{equation*}\mc{G} = \left(\sum_i \mc{E}_i\right)^\vee = \bigcap_i \mc{E}_i^\vee.\qedhere \end{equation*} 
\end{proof}

\begin{remark}\label{rem:choieffros}It follows from \cref{cor:sep} that if $a\in (X\otimes \mathcal S(c))\setminus \mc{G}(c),$ there exists an ${\mathcal S}$-system  $\mathcal E\supseteq \mc{G}$ with a concrete realization, and  $a\notin \mathcal E(c)$. So $a$ can be separated from $\mc{G}$ by a system with a concrete realization, which explains the name \emph{separation theorem}.
A close inspection of the above proof shows that $\mathcal E$ can be chosen to have a concrete realization on $\mathcal A_{c^*}$ in the following way: separate $a$ from $\mathcal G(c)$ with a classical functional as in the classical biduality theorem, and turn this functional into a concrete realization using the techniques from \cref{prop:dual} and \cref{thm:dual}.

Further note that in the theory of aos, the fact that a system is an intersection of $\mathcal S$-hedra is usually formulated as having a concrete realization on an \emph{infinite-dimensional} space (by taking block-diagonal sums of finite realizations). This explains the name \emph{realization theorem}.
\end{remark}

As a corollary of the previous theorem we now prove that a map is completely positive if and only if all of its reductions to (co)intrinsic systems are completely positive. In terms of operator systems, this means that a linear map between two operator systems is completely positive if and only if all of its reductions to matrix spaces are completely positive (with respect to the psd cones).

\begin{corollary}[Completely positive maps reduced to (co)intrinsic systems]\label{kor:mc}
Let $\mc{G},\mc{E}$ be  ${\mathcal S}$-systems on $X$ and $Y$ respectively, and let $\Psi\in{\rm Lin}(X,Y)$. Then $\Psi\in{\rm CP}(\mc{G},\mc{E})$ if and only if for all $c,d\in {\tt C}, \Psi_1\in{\rm CP}(\mathcal B_{c},\mc{G}),\Psi_2\in{\rm CP}(\mc{E},\mathcal A_{d})$  we have $$\Psi_2\circ\Psi\circ\Psi_1\in {\rm CP}(\mathcal B_{c},\mathcal A_{d}).$$ 
\end{corollary}
\begin{proof}
One direction is obvious. For the other, note that \cref{cor:sep} implies that if $\Psi_2\circ\Psi\in \mathrm{CP}(\mc{G},\mathcal A_{d})$ for all the described maps $\Psi_2$, then  $\Psi\in\mathrm{CP}(\mc{G},\mc{E})$. Applying the same argument to the dual of each such $\Psi_2\circ\Psi$, and using biduality, proves the claim. 
\end{proof}

A finite-dimensional version of Arveson's extension theorem (\cref{arveson}) holds in the framework of general conic systems, as we now show. We will see in \cref{ex:theorems} and \cref{sec:app} that the following theorem leads to interesting results when applied to different stems. 

\begin{theorem}[Extension]\label{thm:ext}
Let $X\in {\tt FVec}$ and let $\Theta\colon X\to \mathcal S(c)$ be a concrete realization of the  ${\mathcal S}$-system $\mc{G}$ on $\mathcal A_{c}$. Further assume that either $\Theta$ is onto, or that $\Theta$ hits the interior of $A(c)$ and $A(d)$ is sharp.  
Then for every $\Psi\in{\rm CP}(\mc{G},\mathcal A_{d})$ there exists $\Phi\in{\rm CP}(\mathcal A_{c},\mathcal A_{d})$ with $\Psi=\Phi\circ\Theta$, i.e.\ $\Psi$ is extended to $\Phi$:
$$\xymatrix{\mathcal A_{c} \ar[rr]^{\Phi} && \mathcal A_{d}\\
\mc{G} \ar[urr]^{\Psi}  \ar[u]^{\Theta} &&  }$$ 
\end{theorem}
\begin{proof}
Complete positivity of $\Psi$ means $$\Theta^{-1}(\mathcal A_{c})\subseteq \Psi^{-1}(\mathcal A_{d}).$$ By \cref{thm:selfdual}, after dualizing, this is equivalent to $$( \Psi'\otimes {\rm id}_{\mathcal S(d)})(m_{\mathcal S(d^*)})$$ being in the system generated by $( \Theta'\otimes{\rm id_{\mathcal S(c)}})(m_{\mathcal S(c^*)})$ on $X'$. So there is a sequence $$\Phi_n\in{\rm cone}\left(\mathcal S({\tt C}(c,d))\right)$$ with $$ (\Theta'\otimes \mathcal S(\Phi_n))(m_{\mathcal S(c^*)})\to ( \Psi'\otimes {\rm id}_{\mathcal S(d)})(m_{\mathcal S(d^*)}).$$ This implies  $\mathcal S(\Phi_n)\circ\Theta\to\Psi.$ Since either $\Theta$ is onto, or  $\Theta$ hits the interior of $A(c)$  and $A(d)$ is sharp,  it follows that the sequence $\mathcal S(\Phi_n)$ is bounded and thus contains a convergent subsequence.
\end{proof}

In the classical case, it is well-known from Choi's theorem that, to verify if a map from/into matrices is completely positive, $k$-positivity (\cref{def:cp}) needs to be checked only up to a fixed $k$. Alternatively, one can apply the map to half of a maximally entangled state of a certain fixed dimension and the resulting matrix should be psd. The following is a generalized version thereof.

\begin{theorem}[Completely positive maps from/into (co)intrinsic systems]\label{thm:choi}
Let $\mc{G}$ be an  ${\mathcal S}$-system on $X$, and let  $c\in {\tt C}$.
\begin{enumerate}[label=$(\roman*)$]
\item\label{ins1} For $\Psi\in{\rm Lin}(\mathcal S(c),X)$ the following are equivalent:
\begin{itemize}
\item[(1)] $\Psi\in{\rm CP}(\mathcal B_{c},\mc{G}).$
\item[(2)] $(\Psi\otimes{\rm id}_{\mathcal S(c^*)})(\mathcal B_{c}(c^*))\subseteq \mc{G}(c^*).$
\item[(3)] $(\Psi\otimes{\rm id}_{\mathcal S(c^*)})(m_{\mathcal S(c)})\in \mc{G}(c^*).$
\end{itemize}
\item\label{ins2} For   $\Psi\in{\rm Lin}(X,\mathcal S(c))$ the following are equivalent:
\begin{itemize}
\item[(1)] $\Psi\in{\rm CP}(\mc{G}, \mathcal A_c).$
\item[(2)] $(\Psi\otimes{\rm id}_{\mathcal S(c^*)})(\mc{G}(c^*))\subseteq A(c\otimes c^*).$
\end{itemize}
In particular, if $\psi\in{\rm Pos}(\mc{G}(\mathbf 1^*),A(\mathbf 1^*))$, then $\psi\in{\rm CP}(\mathcal G,\mathcal A_{\mathbf 1})$. 
\end{enumerate}
\end{theorem}
\begin{proof} Immediate  from \cref{thm:selfdual} and \cref{cor:levelmax}.
\end{proof}

Finally, one of the main results of \cite{Fritz17} is that the maximal system over a cone is a spectrahedron and the minimal system over that cone is finitely generated if and only if the base cone is polyhedral. We now prove this more generally with a much shorter proof. 

\begin{theorem}[Systems over a polyhedral cone]\label{thm:poly} Let $\mathcal S$ be a stem for which $B(\mathbf 1)=\mathbb R_{\geqslant 0}.$ Then for a closed convex cone $D\subseteq X$  the following are equivalent:
\begin{enumerate}[label=$(\roman*)$]
\item\label{ma1}  $\overline{\mc{G}}_D$ is an ${\mathcal S}$-hedron.
\item\label{ma2}   $\underline{\mc{G}}_D$ is finitely generated. 
\item\label{ma3}  $D$ is polyhedral.
\end{enumerate}
\end{theorem}
\begin{proof} To show that \ref{ma2} implies \ref{ma3}, assume  $\underline{\mc{G}}_D$ is generated by $a_i\in\underline{\mc{G}}_D(c_i)=D\omin B(c_i)$ for $i=1,\ldots, n$. Writing $$a_i=\sum_{j=1}^{r_i} d_{ij}\otimes b_{ij}$$ with $d_{ij}\in D,b_{ij}\in B(c_i),$ we see that for all $\varphi\in{\tt C}(c_i,\mathbf 1)$ we have $\mathcal S(\varphi)(b_{ij})\in B(\mathbf 1)=\mathbb R_{\geqslant 0}$ and thus $$({\rm id}_X\otimes\mathcal S(\varphi))(a_i)\in {\rm cone}\{d_{i1},\ldots, d_{ir_i}\}.$$ From $\underline{\mc{G}}_D(\mathbf 1)=D\omin B(\mathbf 1)=D$ it follows that $D$ is the convex cone generated by all $d_{ij}$. Thus $D$ is polyhedral. Conversely, it is obvious that \ref{ma3} implies \ref{ma2}. The equivalence of \ref{ma1} and \ref{ma2} follows from \cref{thm:dual}, using that $D^\vee$ is polyhedral if and only if $D$ is.\end{proof}

Let us illustrate the results of this section with our running examples. We will see that the extension theorem (\cref{thm:ext}) becomes the homomorphism theorem for vector systems, and Farkas' lemma for simplex systems.

\begin{example}[Theorems applied to different stems]\label{ex:theorems}\quad
\begin{enumerate}[label=$(\roman*)$]
\item \textbf{(Vector systems)} For vector systems, the extension theorem (\cref{thm:ext}) correspond to the homomorphism theorem for vector spaces. Whenever $\Psi\colon X\to Y$ is a linear map, the projection $X\to X/{\rm ker}(\psi)$ is a surjective concrete realization of the subspace ${\rm ker}(\Psi)$ of $X$. Since $\Psi$ is cp with respect to ${\rm ker}(\Psi)$ and $\{0\}$, ${\rm id}_Y$ is a concrete realization for the subspace $\{0\}$ of $Y$, and thus we can define/extend the map $\Psi$ from the quotient $X/{\rm ker}(\Psi)$ to $Y$.

 \item \textbf{(Simplex systems)} For simplex systems, the separation theorem (\cref{cor:sep}) is the usual separation/bidual theorem for the cone at base level.
The extension theorem yields a version of Farkas' lemma: Given a polyhedral cone, defined by finitely many linear inequalities (which are the coordinate functions of $\Psi_1$), any other positive functional is a positive combination of these inequalities.

 \item \textbf{(Operator systems)} For operator systems, \cref{cor:chkr} is the well-known Choi--Kraus decomposition of a completely positive map as a sum of compression maps. The separation theorem is a combination of the Effros--Winkler Separation and the Choi--Effros realization theorem, as explained in \cref{rem:choieffros}. The extension theorem is the finite-dimensional version of Arveson's extension theorem, and \cref{thm:choi} combines several known results on completely positive maps from/into matrix algebras. \cref{thm:poly} is one of the main results of \cite{Fritz17}.\qedhere
\end{enumerate}
\end{example}

\subsection{Tensor products}\label{ssec:tensorproducts} In this section we introduce tensor products between $\mc{S}$-systems, generalizing the work in \cite{Kavruk} on operator systems. Under the identification $X_1'\otimes X_2 = \mathrm{Lin}(X_1,X_2)$ this will lead to a result on the factorization of maps (\cref{thm:cpdual}) and generalized versions of the Choi--Jamio{\l}kowski isomorphism. We will then prove that for every stem $\mc{S}$, the category ${\tt Sys}(\mc{S})$ is star autonomous itself, when endowed with the minimal tensor product, thus providing a recipe to construct new star autonomous categories and stems in a hierarchical way. 

Throughout this section we assume that \begin{equation}\label{eq:assumption}A(c)\omin A(d)\subseteq A(c\otimes d)\end{equation} holds for all $c,d\in{\tt C}.$ This holds for many of our stems, but not in general, as can be seen for the example that will be introduced in \cref{ssec:TFT}.

A tensor product between two $\mc{S}$-systems is an $\mc{S}$-system on the corresponding tensor product vector space. 
\begin{definition}[Tensor products of $\mc{S}$-systems]\label{def:tp}
 Let $\mc{G}_1,\mc{G}_2$ be  ${\mathcal S}$-systems on $X_1$ and $X_2$, respectively. 
 An  ${\mathcal S}$-system $\mathcal E$ on $X_1\otimes X_2$ is called \emph{a tensor product of $\mc{G}_1$ and $\mc{G}_2$,} if the following two conditions are satisfied for all $c_1,c_2\in{\tt C}$:
\begin{enumerate}
    \item\label{tp1} $\mc{G}_1(c_1)\omin \mc{G}_2(c_2) \subseteq \mathcal E(c_1\otimes  c_2).$
    \item\label{tp2} For  $\Psi_1\in\mathrm{CP}(\mc{G}_1,\mathcal A_{c_1})$ and $\Psi_2\in\mathrm{CP}(\mc{G}_2,\mathcal A_{c_2})$ we have $\Psi_1\otimes\Psi_2\in\mathrm{CP}(\mathcal E, \mathcal A_{c_1\otimes c_2}).$  \qedhere
\end{enumerate}
\end{definition} 

The first condition ensures that, level-wise, the minimal tensor products of the cones in $\mc{G}_1$ and $\mc{G}_2$ is contained in the cones of their tensor product $\mc{E}$. The second condition ensures functoriality of the tensor product with respect to cp maps, which is equivalent to the first condition  for the dual systems, namely that the level-wise minimal tensor products of the cones in $\mc{G}_1^\vee$ and $\mc{G}_2^\vee$ are contained in $\mc{E}^\vee$. 

For any $\mc{G}_1,\mc{G}_2$ there exists a largest tensor product, defined as $$\mc{G}_1\omax\mc{G}_2:=\bigcap_{\scriptsize\begin{array}{c} c_1,c_2\in{\tt C}\\ \Psi_1\in\mathrm{CP}(\mc{G}_1,\mathcal A_{c_1})\\ \Psi_2\in\mathrm{CP}(\mc{G}_2,\mathcal A_{c_2}) \end{array}} (\Psi_1\otimes\Psi_2)^{-1}(\mathcal A_{c_1\otimes c_2}).$$ 
By construction the maximal tensor product satisfies \cref{def:tp}\ref{tp2}. In order to verify \cref{def:tp}\ref{tp1} the assumption \eqref{eq:assumption} is needed. If this assumption is not fulfilled, then  tensor products of certain systems (for example of intrinsic systems already) might not exist.

Since intersections of  ${\mathcal S}$-systems are  ${\mathcal S}$-systems, existence of the largest tensor product also implies existence of a smallest tensor product $\mc{G}_1\omin \mc{G}_2,$ which is the system generated by all sets $\mc{G}_1(c_1)\omin \mc{G}_2(c_2)$. For (co)intrinsic systems the smallest and largest tensor product can be determined explicitly as follows.

\begin{theorem}[Tensor products of (co)intrinsic systems]\label{thm:intens} For $c_1,c_2\in{\tt C}$ we have $$\mathcal B_{c_1}\omin\mathcal B_{c_2}=\mathcal B_{c_1\check{\otimes} c_2} \ \mbox{ and } \ \mathcal A_{c_1}\omax \mathcal A_{c_2}= \mathcal A_{c_1\otimes c_2}.$$
\end{theorem}
\begin{proof} Since $\mathcal B_{c_i}$ is generated by $m_{\mathcal S(c_i)}\in\mathcal S(c_i)\otimes\mathcal S(c_i^*)$ and $$m_{\mathcal S(c_1)}\otimes m_{\mathcal S(c_2)}\in \mathcal S(c_1)\otimes \mathcal S(c_1^*)\otimes\mathcal S(c_2)\otimes\mathcal S(c_2^*)=\mathcal S(c_1\ot c_2)\otimes\mathcal S((c_1\ot c_2)^*)$$ we obtain 
the first statement. 
From  \cref{cor:chkr} we obtain  inclusion from left to right in the second  equation.
Since ${\rm id}_{\mathcal S(c_i)}\in \mathrm{CP}(\mathcal A_{c_i},\mathcal A_{c_i})$,  we  obtain ${\rm id}_{\mathcal S(c_1)}\otimes {\rm id}_{\mathcal S(c_2)}={\rm id}_{\mathcal S(c_1\otimes c_2)}\in\mathrm{CP}(\mathcal A_{c_1}\omax\mathcal A_{c_2},\mathcal A_{c_1\otimes c_2})$ and thus  $\mathcal A_{c_1}\omax\mathcal A_{c_2}\subseteq \mathcal A_{c_1\otimes c_2}.$
\end{proof}

The following theorem collects several results on the tensor product of $\mc{S}$-systems. For example, \ref{tpdu3} shows that the dual of a minimal tensor product of systems is equal to the maximal tensor product of the dual systems, and \ref{tpdu4} shows that the minimal tensor product of minimal systems is equal to the minimal system over the minimal tensor product of the respective cones on the base levels. 

\begin{theorem}[Properties of tensor products of  $\mc{S}$-systems]\label{thm:tensorpr} Let $\mc{G}_1,\mc{G}_2, \mathcal E_1,\mathcal E_2$ be  ${\mathcal S}$-systems on $X_1, X_2, Y_1, Y_2$ respectively, and $D_1\subseteq X_1,D_2\subseteq X_2$ closed cones. Then we have
\begin{enumerate}[label=$(\roman*)$]
\item\label{tpdu1} $\mc{G}_1(\mathbf 1)\omin \mc{G}_2(\mathbf 1)\subseteq (\mc{G}_1\omin \mc{G}_2)(\mathbf 1).$
\item\label{tpdum}If $B(\mathbf 1)=\Rnn,$ then  $(\mc{G}_1\omax\mc{G}_2)(\mathbf 1^*)\subseteq\mc{G}_1(\mathbf 1^*) \omax\mc{G}_2(\mathbf 1^*).$
\item\label{tpmima}The inclusions in \ref{tpdu1} and \ref{tpdum} can be strict.
\item\label{tpdu2} $(\mc{G}_1\omin \mc{G}_2)^{\vee}= \mc{G}_1^{\vee}\omax \mc{G}_2^{\vee}$ and $(\mc{G}_1\omax\mc{G}_2)^{\vee}= \mc{G}_1^{\vee}\omin  \mc{G}_2^{\vee}.$
 \item\label{tpdu3} $\barbeloww{\mc{G}}_{D_1}\omin \barbeloww{\mc{G}}_{D_2}=\barbeloww{\mc{G}}_{D_1\omin D_2}$ and $\bar{\mc{G}}_{D_1}\omax\bar{\mc{G}}_{D_2}=\barabove{\mc{G}}_{D_1\omax D_2}.$
\item\label{tpdu4} Minimal and maximal tensor products are functorial, i.e.\ if $\Psi_1\in \mathrm{CP}(\mc{G}_1,\mathcal E_1)$ and $\Psi_2\in \mathrm{CP}(\mc{G}_2,\mathcal E_2)$, then $\Psi_1\otimes\Psi_2\in \mathrm{CP}(\mc{G}_1\omin \mc{G}_2,\mathcal E_1\omin \mathcal E_2)\cap \mathrm{CP}(\mc{G}_1\omax \mc{G}_2,\mathcal E_1\omax \mathcal E_2).$
\end{enumerate}
\end{theorem}
\begin{proof}  \ref{tpdu1} holds by definition. For \ref{tpdum} take $\psi_i\in\mc{G}_i(\mathbf 1^*)^\vee$. Then $\psi_i\in\mathrm{CP}(\mc{G}_i,\mathcal A_{\mathbf 1})$ by \cref{thm:choi}, thus $\psi_1\otimes\psi_2\in\mathrm{CP}(\mc{G}_1\omax\mc{G}_2,\mathcal A_{\mathbf 1}),$ so $\psi_1\otimes\psi_2\geqslant 0$ on $\mathcal(\mc{G}_1\omax\mc{G}_2)(\mathbf 1^*).$  For \ref{tpmima} see \cref{ex:tens}\ref{ex:tensuneq2}  and also  \cref{ssec:conesystems} below.
For \ref{tpdu2} we compute
\begin{align*}
(\mc{G}_1\omax\mc{G}_2)^{\vee}&=\left(\bigcap_{\scriptsize c_1,c_2,\Psi_1, \Psi_2} (\Psi_1\otimes\Psi_2)^{-1}(\mathcal A_{c_1\otimes c_2})\right)^{\vee}\\
&=\sum_{\scriptsize c_1,c_2, \Psi_1, \Psi_2} \left((\Psi_1\otimes\Psi_2)^{-1}(\mathcal A_{c_1\otimes c_2})\right)^{\vee} \\
&=\sum_{\scriptsize c_1,c_2, \Psi_1', \Psi_2'} (\Psi_1'\otimes\Psi_2')(\mathcal B_{(c_1\otimes c_2)^*}) \\
&=\sum_{\scriptsize c_1,c_2, \Psi_1', \Psi_2'} (\Psi_1'\otimes\Psi_2')(\mathcal B_{c_1^*}\omin \mathcal B_{c_2^*}) \\
&= \sum_{\scriptsize c_1,\Psi_1'} \Psi_1'(\mathcal B_{c_1^*})\omin \sum_{\scriptsize c_2,\Psi_2'} \Psi_2'(\mathcal B_{c_2^*}) \\ &= \left(\bigcap_{\scriptsize c_1,\Psi_1} \Psi_1^{-1}(\mathcal A_{c_1})\right)^{\vee}\omin\left(\bigcap_{\scriptsize c_2,\Psi_2} \Psi_2^{-1}(\mathcal A_{c_2})\right)^{\vee
}
\\
&=\mc{G}_1^{\vee}\omin \mc{G}_2^{\vee}
\end{align*} where the last equality comes from the realization theorem. The other statement in \ref{tpdu2} follows directly from biduality.

In the first equation of \ref{tpdu3} we obtain $\supseteq$ from \ref{tpdu1}. The other inclusion follows from the observation that $\barbeloww{\mc{G}}_{D_1\omin  D_2}$ contains all $$(D_1\omin B(c_1))\omin (D_2\omin B(c_2))=(D_1\omin D_2)\omin(B(c_1)\omin B(c_2))\subseteq (D_1\omin D_2)\omin B(c_1\otimes c_2),$$ using \cref{lem:ext}\ref{sys1}. The second equation follows from duality.
In \ref{tpdu4}, the containment of $\Psi_1\otimes\Psi_2$ in the first set is obvious, while the containment in the second follows from duality. 
\end{proof}

\begin{example}[Tensor products of $\mc{S}$-systems]\label{ex:tens}\quad
\begin{enumerate}[label=$(\roman*)$]
\item\label{extp1}\textbf{(Vector systems)} For two vector systems that correspond to subspaces $U_1\subseteq X_1$ and $U_2\subseteq X_2$, the minimal tensor product corresponds to their usual tensor product $$U_1\otimes U_2.$$ The maximal tensor product however corresponds to $$U_1\otimes X_2+ X_1\otimes U_2.$$

\item\label{extp2}\textbf{(Simplex systems)} For simplex systems, the minimal and maximal tensor product correspond to precisely those concepts for their base cones.
\item\label{ex:tensuneq2}\textbf{(Operator systems)} Since operator systems rely on a compact closed and self-dual stem, we obtain from \cref{thm:intens} $$\mathcal A_{H_1}\omin \mathcal A_{H_2}=\mathcal A_{H_1}\omax \mathcal A_{H_2}= \mathcal A_{H_1\otimes H_2},$$ which is just the system of positive semidefinite operators, and  has $\textrm{Psd}(H_1\otimes H_2)$ as base cone. However, the minimal and maximal tensor product of the levels do not coincide, since  \begin{align*}\mathcal A_{H_1}(\mathbb R)\omin\mathcal A_{H_2}(\mathbb R)&=\textrm{Psd}(H_1)\omin \textrm{Psd}(H_2)\\ &\subsetneq \textrm{Psd}(H_1\otimes H_2) \\ &\subsetneq \textrm{Psd}(H_1)\omax \textrm{Psd}(H_2)= \mathcal A_{H_1}(\mathbb R)\omax\mathcal A_{H_2}(\mathbb R).\qedhere\end{align*}
\end{enumerate}
\end{example}

Under the identification of tensors and linear maps, the base level of any tensor product of $\mc{S}$-systems on $X'$ and $Y$ can be interpreted as a subset of $\mathrm{Lin}(X,Y)$. For example, it is well-known that the set of entanglement breaking maps $\textrm{EB}(C,D)\coloneq C^\vee \omin D$ between cones $C\subseteq X,D\subseteq Y$ is given by those linear maps from $X$ to $Y$ that can be factorized positively through a simplex cone. The following theorem generalizes this result and characterizes the subset of linear maps for two special cases. 

\begin{theorem}[Completely positive maps and factorization]\label{thm:cpdual}
Let $\mc{G},\mc{E}$ be  ${\mathcal S}$-systems on $X$ and $Y,$ respectively. 
\begin{enumerate}[label=$(\roman*)$]
\item\label{sl1} Under the identification $X'\otimes Y={\rm Lin}(X,Y)$  we have $$(\mc{G}^{\vee}\omax\mc{E})(\mathbf 1)=\mathrm{CP}(\mc{G},\mc{E}).$$
 \item\label{sl2}   $(\mc{G}^{\vee}\omin\mc{E})(\mathbf 1^*)$ equals the closed conic hull  of the maps $\Psi\colon X\to Y$ which admit a factorization $$\Psi=\Psi_2\circ \Psi_1$$ with $\Psi_1\in\mathrm{CP}(\mc{G},\mathcal A_{c})$ and $\Psi_2\in\mathrm{CP}(\mathcal B_{c},\mc{E}).$
\end{enumerate}
\end{theorem}
\begin{proof}\ref{sl1}: For $\Psi\in\mathrm{Lin}(X,Y)=X'\otimes Y,$ the fact that $\Psi\in (\mc{G}^{\vee}\omax\mc{E})(\mathbf 1)$ is equivalent to   $$(\Psi_1\otimes(\Psi_2\circ\Psi))(m_{X'})\in\mathcal A_{d\otimes e}(\mathbf 1)=A(d\otimes e)$$ for all $d,e\in{\tt C}, \Psi_1\in \mathrm{CP}(\mc{G}^{\vee}, \mathcal A_{d}), \Psi_2\in \mathrm{CP}(\mc{E}, \mathcal A_{e})$. Since $(\Psi_1\otimes(\Psi_2\circ\Psi))(m_{X'})$ corresponds to $$(\Psi_2\circ\Psi\circ\Psi_1'\otimes {\rm id}_{\mathcal S(d)})(m_{\mathcal S(d^*)})$$ under symmetry, we conclude from  \cref{thm:choi} that this is equivalent to $\Psi_2\circ\Psi\circ\Psi_1'\in \mathrm{CP}(\mathcal B_{d^*},\mathcal A_{e})$, which is equivalent to $\Psi\in\mathrm{CP}(\mc{G},\mc{E})$ by \cref{kor:mc}.

\ref{sl2}:  For $\widetilde \Psi_1\in\mc{G}^{\vee}(d)=\mathrm{CP}(\mc{G},\mathcal A_{d}),  \widetilde\Psi_2\in\mc{E}(e)=\mathrm{CP}(\mc{E}^{\vee},\mathcal A_{e}),$ and 
$\varphi\in{\tt C}(d\otimes e,\mathbf 1^*)={\tt C}(d,e^*),$  the element $$({\rm id}_{X'\otimes Y}\otimes\mathcal S(\varphi))(\widetilde\Psi_1\otimes \widetilde\Psi_2)\in (\mc{G}^{\vee}\omin \mc{E})(\mathbf 1^*)$$
corresponds to the map $\widetilde\Psi_2'\circ\mathcal S(\varphi)\circ\widetilde\Psi_1$. With $c:=d^*, \Psi_1:=\mathcal S(\varphi)\circ\widetilde\Psi_1$ and $\Psi_2:=\widetilde\Psi_2'$ we obtain the desired factorization. Conversely, let $\Psi=\Psi_2\circ \Psi_1$ be such a factorization. Then for $\Psi_1\in\mathrm{CP}(\mc{G},\mathcal A_c)=\mc{G}^{\vee}(c), \Psi_2'\in \mathrm{CP}(\mathcal B_c,\mc{E})=\mc{E}(c^*)$ we obtain \begin{equation*}\Psi=({\rm id}_{X'\otimes Y}\otimes\textrm{ev})(\Psi_1\otimes\Psi_2')\in(\mc{G}^{\vee}\omin\mc{E})(\mathbf 1^*).\qedhere\end{equation*} 
\end{proof}

Applying this theorem to the cointrinsic and intrinsic system, we see that the completely positive maps between these systems are equal to the intrinsic cone at the tensor product level. 

\begin{corollary}
With respect to the identification $$\mathrm{Lin}(\mathcal S(c),\mathcal S(d))=\mathcal S(c)'\otimes \mathcal S(d)=\mathcal S(c^*\otimes d)$$ we have $\mathrm{CP}(\mathcal B_{c},\mathcal A_{d})=A(c^*\otimes d).$
\end{corollary}
\begin{proof}  By \cref{thm:cpdual} we have \begin{equation*}\mathrm{CP}(\mathcal B_{c},\mathcal A_{d})=(\mathcal B_{c}^{\vee}\omax \mathcal A_{d})(\mathbf 1)=(\mathcal A_{c^*}\omax\mathcal A_{d})(\mathbf 1)=\mathcal A_{c^*\otimes d}(\mathbf 1)=A(c^*\otimes d).\qedhere\end{equation*} 
\end{proof}

\begin{example}[Completely positive maps on $\mc{S}$-systems] \quad
\begin{enumerate}[label=$(\roman*)$]
\item\textbf{(Vector systems)} If the vector systems $\mc{G}_1,\mc{G}_2$ correspond to the subspaces $U_1\subseteq X_1,U_2\subseteq X_2$, cp maps are those linear maps between $X_1$ and $X_2$ that map $U_1$ into $U_2$. \cref{thm:cpdual}\ref{sl1} says that these are of the form $\Psi_1+\Psi_2,$ where $\Psi_1$ maps $U_1$ to $\{0\}$ and $\Psi_2$ maps $X_1$ to $U_2.$ \ref{sl2} says that elements of the minimal tensor product are compositions $\Psi_2\circ\Psi_1$ of maps with $\Psi_1(U_1)=\{0\}$ and $\Psi_2(V)\subseteq U_2$, which are precisely the maps with total image $U_2$ and $U_1$ in the kernel.
\item\textbf{(Simplex systems)} For simplex systems, \cref{thm:cpdual} corresponds to the well-known statements that $\mathrm{Pos}(C,D)$ equals $C^\vee\omax D$, and $C^\vee\omin D$ corresponds to those maps that admit a positive factorization through a simplex cone, also known as entanglement breaking maps. This follows since intrinsic and cointrinsic systems coincide, and correspond to simplex cones.
\item\textbf{(Operator systems)} Since the operator stem is also compact closed and self-dual, the minimal tensor product at base level consists of those cp maps that admit a cp factorization through some $\mathsf B(H)$, equipped with the usual operator system of positive semidefinite operators.\qedhere
\end{enumerate}
\end{example}

Finally we come to the climax of this section. For every stem $\mc{S}$ we have defined  ${\tt Sys}(\mc{S})$ as the category of all $\mc{S}$-systems and completely positive maps. In \cref{ssec:duality} we introduced a dual notion of $\mc{S}$-systems and in \cref{ssec:tensorproducts} a tensor product between them. Now we show that this notion of duality and tensor product are compatible and that ${\tt Sys}(\mc{S})$ is itself a star autonomous category, from which we can build a new stem. Repeating this procedure leads to a tree of new stems on top of every $\mc{S}$ and is a source of many new examples of star autonomous categories. 

\begin{theorem}[A tree of stems]\label{thm:hierarchy} For every stem $\mathcal S$, $({\tt Sys}(\mathcal S), \omin, \mathcal B_{\mathbf 1^*},{\vee})$ is a star autonomous category, and the functor \begin{align*}\mathcal T\colon {\tt Sys}(\mathcal S)&\to {\tt FVec} \\ (\mc{G},X)&\mapsto X \end{align*} defines a valid stem.
\end{theorem}
\begin{proof}
To verify that $\omin$ defines a symmetric monoidal structure with $\mathcal B_{\mathbf 1^*}$ as the monoidal unit, we need to show that for all $c\in {\tt C}$ 
$$(\mc{G} \omin \mc{B}_{\mathbf 1^*})(c) = \mc{G}(c).$$
First, $$\mc{G}(c) \omin \mc{B}_{\mathbf 1^*}(\mathbf 1) = \mc{G}(c) \omin B(\mathbf 1) \supseteq \mc{G}(c),$$ so $\mc{G}(c) \subseteq (\mc{G} \omin \mc{B}_{\mathbf 1^*})(c)$. On the other hand, by \cref{lem:ext}, $$\mc{G}(c) \omin B(d) \subseteq \mc{G}(c \otimes d),$$ which shows that by taking a tensor product with $\mathcal B_{\mathbf 1^*}$ we do not obtain more than elements from $\mc{G}$, so $(\mc{G} \omin \mc{B}_{\mathbf 1^*})(c) \subseteq \mc{G}(c)$.

For $\mathcal S$-systems $\mc{G},\mathcal E,\mathcal D$ we have $$(\mc{G}\omin\mathcal E)^{\vee}\omax\mathcal D=\mc{G}^{\vee}\omax(\mathcal E^{\vee}\omax\mathcal D).$$ At level $\mathbf 1$ this means   $\mathrm{CP}(\mc{G}\omin\mathcal E,\mathcal D)=\mathrm{CP}(\mc{G}, \mathcal E^\vee \omax\mathcal D),$  which is the star autonomous condition (\cref{def:star}). All the remaining properties one needs to check are straightforward.
\end{proof}

\begin{example}[A tree of stems] Let us go through our running examples. 
\begin{enumerate}[label=$(\roman*)$]
\item\textbf{(Vector systems)} The category of vector systems is equivalent to the category of subspaces: $${\tt Sys(FVec)}={\tt SubSpace}.$$ Its objects are pairs $(U,X)$ where $X$ is a finite-dimensional real vector space and $U$ is a subspace of $X$, and the morphisms are the linear maps mapping the respective subspaces into each other. The tensor product is defined as $$(U_1,X_1)\otimes (U_2,X_2):=(U_1\otimes U_2,X_1\otimes X_2),$$ and the duality by  $$(U,X)^*:=(U^\perp, X').$$ The monoidal unit is $\mathbf 1=(\mathbb R,\mathbb R),$ which is not isomorphic to the counit $\mathbf 1^*=(\{0\},\mathbb R')$.  The functor $$\mathcal T\colon {\tt Subspace}\to{\tt FVec}; \ (U,X)\mapsto X$$ makes it a self-dual stem with $A(U) = U =B(U)$ that is not compact closed.

\item\textbf{(Simplex systems)} For the category of simplex systems we have  $${\tt Sys(SCone)}={\tt FCone},$$ and the arising new stem is studied in \cref{ssec:conesystems} in detail. 

\item\textbf{(Operator systems)}  ${\tt Sys}({\tt FHilb^{op}})$ is the category of operator systems and completely positive maps.  The monoidal product is the minimal tensor product of operator systems with monoidal unit $\mc{B}_{1^*} = \mc{PSD}_1$, the operator system from \cref{ex:aos} on $\mathsf B(\C)_{\rm her} = \R$ with the psd cone at every level
$$\mc{PSD}_1(s) = \mathrm{Psd}(s) .$$
${\tt Sys}({\tt FHilb^{op}})$ is not compact closed, since the minimal and maximal tensor product of operator systems do not coincide, even though the unit and counit do. 

The functor $$\mc{T}\colon {\tt Sys}({\tt FHilb^{op}}) \to {\tt FVec}; (\mc{G},X) \mapsto X,$$
acting on morphisms as the identity, defines a stem that is not compact closed. To compute the (co)intrinsic cones we use that duality (\cref{prop:dual}) gives us $\mathrm{CP}(\mc{G}, \mc{PSD}_1^\vee) = \mc{G}(1)^\vee$ and that by \cref{thm:choi} we have $$\mathrm{CP}(\mc{PSD}_1, \mc{G}) = \mc{G}(\mathbf{1})$$ to see that
$$A(\mc{G}) = \mc{G}(\mathbf{1}) = B(\mc{G}),$$
so the stem is self-dual.\qedhere
\end{enumerate}
\end{example}

\section[Applications]{More stems and applications}\label{sec:app}
In this final section we introduce five more examples of stems from a variety of areas in mathematics, and apply the theory that has been established in \cref{sec:theory} to them. First, we consider a stem of general cones, leading to the definition of cone systems\footnote{Note: Not to be confused with general \emph{conic} systems.} (\cref{ssec:conesystems}). Then we introduce two-sided systems (\cref{ssec:two-sided}) and generalize results on mapping cones. In \cref{ssec:conic_groups} we study a refinement of cone systems with a group action, and in \cref{ssec:operator_groups} a refinement of operator systems with a unitary group action. Finally, topological quantum field theories (TFT) can be seen as a stem for which we study TFT-systems (\cref{ssec:TFT}). For all examples in this section we highlight interesting results that follow from the theorems proven in \cref{ssec:theorems}. 

\subsection{Cone systems}\label{ssec:conesystems}

Let us take {\tt C} to be the category {\tt FCone} of finite-dimensional closed convex cones and positive linear maps. The monoidal structure is given by the minimal tensor product $C\omin D$ of cones, the monoidal unit is $\Rnn$, and ${}^*$ is defined by taking dual cones (and maps): $$C^*:=C^\vee.$$  This is a star autonomous category where $\check{\otimes}$ is the maximal tensor product of cones\footnote{Choosing the maximal tensor product as the first monoidal structure instead of the minimal is not possible, as it would require entanglement breaking maps as morphisms instead of positive maps, i.e maps for which $\mathrm{EB}(C,D) = C^\vee \omin D$, to make it star autonomous. But this variation of {\tt FCone} is not a category, since  the identity map $\textrm{id}\colon V \to V$ is not part of the morphism set in general.}. In particular, {\tt FCone} is not compact closed.
As the stem $\mathcal S\colon {\tt FCone}\to{\tt FVec}$ we take the forgetful functor that `forgets' the cones while remembering the spaces and  morphisms, as introduced in \cref{ex:saf}, and obtain $$B(C)=C=A(C).$$ This is a self-dual stem that is not compact closed. We call the arising $\mathcal S$-systems \emph{cone systems}.  In other words, we are now studying the category $${\tt Sys(FCone)=\tt Sys(Sys(SCone))}.$$
For each closed convex cone $D\subseteq X$ and each cone system $\mc{G}$ over $D$, we have $$\underline{\mc{G}}_D(C)=D\omin C\subseteq\mc{G}(C)\subseteq D\omax C=\overline{\mc{G}}_D(C)$$ for all cones $C\in{\tt FCone}.$ We can thus understand $\mc{G}$ as a ``conic tensor product from the left with $D$", and write $$D\otimes_{\mc{G}}C:=\mc{G}(C)$$ to emphasize this fact. The smallest and the largest system over $D$ correspond to the minimal and maximal tensor product, which are also exactly the (co)intrinsic systems:  \begin{align*}
\mathcal{A}_{D}(C)&= D\omin C= \underline{\mc{G}}_D(C)\\
    \mathcal B_{D}(C)&= D\omax C=\overline{\mc{G}}_D(C).
\end{align*} A cone system has a concrete realization if it is the inverse image of a minimal system, and since one can reformulate finite intersections in terms of direct sums, these are also the cone-hedra. In particular, each minimal system has a concrete realization (in contrast to operator systems, where many minimal systems do actually \emph{not} have a concrete realization \cite{Fritz17}).

 For a cone system $\mc{G}$ on $X$, if we write $\mc{G}(C)=D\otimes_{\mc{G}} C$,   we have $$(D\otimes_{\mc{G}} C^\vee)^\vee=\mc{G}^\vee(C)= D^\vee \otimes_{\mc{G}^\vee} C,$$ so the dual system describes a dual tensor product to $\otimes_{\mc{G}}.$  
 

The separation theorem (\cref{cor:sep}) and \cref{thm:poly} entail the following:
\begin{corollary}[Left tensor products of cones] For $D\in {\tt FCone}$ let $$D\otimes_{\mc{G}}\cdot \colon {\tt FCone} \to {\tt FCone}$$ be a tensor product from the left with $D$, i.e.\ a functor such that $$D\omin C\subseteq D\otimes_{\mc{G}}C \subseteq D\omax C$$ for all $C\in{\tt FCone}.$ Then there are $D_i\in{\tt FCone}$ and  $\Psi_i\in\mathrm{Pos}(D,D_i)$ such that for all $C\in{\tt FCone}$ we have $$D\otimes_{\mc{G}}C=\bigcap_i (\Psi_i\otimes {\rm id})^{-1}(D_i\omin C).$$ The maximal tensor product $D\omax\cdot$ is a finite such intersection (i.e. a single inverse image) if and only if $D$ is polyhedral.
\end{corollary}

The extension theorem, combined with \cref{thm:choi}, gives a vector valued version of the finite-dimensional conic Hahn--Banach theorem, sometimes also named after Riesz or Krein.
\begin{corollary}[Vector valued Riesz--Krein theorem]\label{cor:riesz}
Let  $C\subseteq V, D\subseteq W$ be finite-dimensional cones with $D$ sharp. Let further   $U\subseteq V$ be a subspace that intersects the interior of $C$. Then $$\Psi\in \mathrm{Pos}(C\cap U,D)\subseteq \mathrm{Lin}(U,V)$$ has an extension to an element of $\mathrm{Pos}(C,D)$ if and only if $\Psi\otimes{\rm id}_{W'}$ maps  $$(C\omin D^\vee)\cap (U\otimes W')$$ into $D\omin D^\vee.$
In case that $D=\mathbb R_{\geqslant 0}$,  this is true for every positive functional, which gives the usual formulation of the theorem.
\end{corollary}

From \cref{thm:intens} and \cref{thm:tensorpr} we obtain that for any two minimal cone systems, all their tensor products coincide, as well as for two maximal cone systems: \begin{align*}\barbeloww{\mc{G}}_{D}\omin\barbeloww{\mc{G}}_{E}&= \barbeloww{\mc{G}}_{D}\omax\barbeloww{\mc{G}}_{E}\\ \bar{\mc{G}}_{D}\omin\bar{\mc{G}}_{E}&= \bar{\mc{G}}_{D}\omax\bar{\mc{G}}_{E}.\end{align*} Here, the former system has $D\omin E$ as base level and the latter $D\omax E.$ For any two cone systems over $D$ and $E$, both their minimal and maximal tensor product can have a base cone in between $D\omin E$ and $D\omax E.$

Finally, for cone systems, the maps from \cref{thm:cpdual}\ref{sl2} are those that admit a cp factorization through some $\barbeloww{\mc{G}}_C\subseteq\barabove{\mc{G}}_C.$

\subsection[Two-sided systems]{Two-sided systems and  free mapping cones}\label{ssec:two-sided}
In this section we describe the connection between mapping cones and abstract operator systems with an additional symmetry, as in  \cite{Johnston_2012}. We show that instead of adding this symmetry we can describe free mapping cones with a different stem, which proves to be very fruitful when describing tensor products. 
A \emph{mapping cone} is a closed convex cone $$M\subseteq \mathrm{Pos}(\mathrm{Psd}(H_1),\mathrm{Psd}(H_2))$$ that is closed under composition with cp maps (in the classical sense) from the left and from the right. If it is only closed under composition with cp maps from the left, it is called \emph{left-cp-invariant}, or just a \emph{left mapping cone} (similarly with right). Well-known examples of mapping cones are the cone of entanglement breaking maps, the cone of completely positive maps and the cone of all positive maps. 

Now assume that $\mc{G}$ is an  operator system on $\mathsf B(H_1),$ with the cone $\mathrm{Psd}(H_1)$ at base level. From $$\mathrm{Psd}(H_1)\omin\mathrm{Psd}(H_2)\subseteq\mc{G}(H_2)\subseteq\mathrm{Psd}(H_1)\omax\mathrm{Psd}(H_2)=\mathrm{Pos}( \mathrm{Psd}(H_1),\mathrm{Psd}(H_2))$$ and the definition of operator systems it is clear that $M=\mc{G}(H_2)$ is a left mapping cone. Conversely, each left mapping cone $M$ arises as the $H_2$-level of such an operator system, namely just $\mc{G}=\mathcal W(M),$ the operator system generated by $M$. This was also proven as part of Theorem 6 in  \cite{Johnston_2012}, and  can be taken as a justification for referring to operator systems over a cone ${\rm Psd}(H_1)$ as \emph{free}  left mapping cones instead.

Mapping cones arise similarly as the higher levels of operator systems over $\mathrm{Psd}(H_1)$ with an additional symmetry, namely the condition that all maps which are cp on $\mathsf B(H_1)$ (in the classical sense) are also cp with respect to the operator system. Operator systems with this symmetry are called \emph{super homogeneous} in \cite{Johnston_2012}, but will be called \emph {free mapping cones} here. 

Now we will demonstrate that this symmetry can be replaced by using a stem other than the operator stem. Since this is easier to digest in a bit  more generality, we first assume that $\mathcal S_1\colon {\tt C}_1\to{\tt FVec}$ and $\mathcal S_2\colon {\tt C}_2\to{\tt FVec}$ are two arbitrary stems. The product category $${\tt C}={\tt C}_1 \times {\tt C}_2$$ is again star autonomous,  and compact closed if ${\tt C_1}$ and ${\tt C}_2$ are. We now consider  the functor
\begin{align*}\mathcal S_1\otimes\mathcal S_2\colon {\tt C}&\to{\tt FVec}\\  
(c_1,c_2)&\mapsto \mathcal S_1(c_1)\otimes\mathcal S_2(c_2)= \mathrm{Lin}(\mathcal S_1(c_1^*),\mathcal S_2(c_2)) \\ (\Phi_1,\Phi_2)&\mapsto \mathcal S_1(\Phi_1)\otimes\mathcal S_2(\Phi_2)= \left[\psi \mapsto \mathcal S_2(\Phi_2)\circ\psi\circ\mathcal S_1(\Phi_1^*)\right]
\end{align*} which defines a valid stem. We find
\begin{align*}
B(c_1,c_2)&=B_1(c_1)\omin B_2(c_2)\\ &= \left\{ \Psi\in\mathrm{Pos}(A_1(c_1^*),B_2(c_2))\mid \Psi \mbox{ factors  through a simplex cone} \right\} \\& =:\textrm{EB}(A_1(c_1^*),B_2(c_2))\\ A(c_1,c_2)&=A_1(c_1)\omax A_2(c_2)\\&=\mathrm{Pos}(B_1(c_1^*),A_2(c_2)).
\end{align*}  
We call systems over such a  stem \emph{two-sided systems}. Now assume $\mc{G}$ is an  $\mathcal S_1\otimes\mathcal S_2$-system on the space $\mathbb R,$ with $\mc{G}(\mathbf 1)=\mathbb R_{\geqslant 0}$ and $\mc{G}(\mathbf 1^*)=\mathbb R_{\geqslant 0}^\vee.$ Then for all $(c_1,c_2)\in {\tt C}_1\times{\tt C}_2$ we have $$\textrm{EB}(A_1(c_1^*),B_2(c_2))\subseteq \mc{G}(c_1,c_2)\subseteq {\rm Pos}(B_1(c_1^*),A_2(c_2))$$ and the family of all $\mc{G}(c_1,c_2)$ is closed under compositions with maps coming from ${\tt C}_1$ and ${\tt C}_2$  from the right and from the left, respectively. In the special case where $\mathcal S_1=\mathcal S_2$ are both the operator stem, we thus obtain mapping cones $$\textrm{EB}({\rm Psd}(H_1),{\rm Psd}(H_2))\subseteq\mc{G}(H_1,H_2)\subseteq \mathrm{Pos}({\rm Psd}(H_1),{\rm Psd}(H_2))$$ and each mapping cone appears as a level of such a system.

\begin{definition}[Free mapping cone]A \emph{free mapping cone} is a two-sided operator system over $\mathbb R_{\geqslant 0}$. The system  $\mathcal B_{\mathbf 1^*}=:\mathcal{EB}$ of all entanglement breaking maps is the smallest free mapping cone. The system  $\mathcal A_1=:\mathcal{POS}$ of all positive maps is the largest free mapping cone.
\end{definition}

The separation theorem (\cref{cor:sep}) results in the following characterization of (free) mapping cones. 

\begin{corollary}\label{cor:mpcr}
(i) Let $\mc{G}$ be a free mapping cone. Then there is a family $\left(\Psi_i\right)_{i\in I}$ of positive maps $\Psi_i\in\mathrm{Pos}({\rm Psd}(H_{i1}),{\rm Psd}(H_{i2}))$ such that $$\Phi\in \mc{G}\  \Leftrightarrow\  \Psi_i\otimes \Phi \mbox{ is a positive map for all }  i\in I$$
holds for all $H_1,H_2\in{\tt FHilb}$ and all $\Phi\in\mathrm{Lin}(\mathsf B(H_1)_{\rm her},\mathsf B(H_2)_{\rm her}).$ 

($ii$) Let $M\subseteq  \mathrm{Pos}({\rm Psd}(H_1),{\rm Psd}(H_2))$ be a (classical) mapping cone. Then there is a family $\left(\Psi_i\right)_{i\in I}$ of positive maps $\Psi_i\in\mathrm{Pos}({\rm Psd}(H_1),{\rm Psd}(H_2))$ such that $$\Phi\in M\  \Leftrightarrow\  \Psi_i\otimes \Phi \mbox{ is a positive map for all }  i\in I$$ holds for all $\Phi\in\mathrm{Lin}(\mathsf B(H_1)_{\rm her},\mathsf B(H_2)_{\rm her}).$ 
\end{corollary}
\begin{proof} This is just the realization theorem, combined with \cref{rem:choieffros} for the bound on the level in ($ii$).
\end{proof}

\begin{example}[Free mapping cones]\quad
\begin{enumerate}[label=$(\roman*)$]
\item Using $\Psi={\rm id}_{\mathsf B(H)}$ for a concrete realization as in \cref{cor:mpcr} gives rise to the free mapping cone of all $H$-positive maps (also called $k$-positive in \cref{def:cp}, where $k=\dim H$). The dual system, generated by ${\rm id}_{\mathsf B(H)},$ consists of maps that admit a cp-factorization through $\mathsf B(H)$.
\item Using $\Psi={\rm id}_{\mathsf B(H)}$ for \emph{all} $H\in{\tt FHilb}$ defines the free mapping cone of all cp-maps $$\mathcal{CP}:=\left(\mathrm{CP}(\mathsf B(H_1),\mathsf B(H_2)) \right)_{H_1,H_2}$$ which is also generated by all ${\rm id}_{\mathsf B(H)}$ and thus self-dual (caution: cp is meant in the sense of classical operator systems here!).
\item A map $\Phi\in \mathrm{Lin}(\mathsf B(H_1)_{\rm her},\mathsf B(H_2)_{\rm her})$ is called \emph{$H$-entanglement breaking} (also $k$-entanglement breaking in the literature, where $k=\dim H$) if $${\rm id}_{\mathsf B(H)}\otimes\Phi$$ maps ${\rm Psd}(H\otimes H_1)$ to ${\rm Psd}(H)\omin{\rm Psd}(H_2)$.  The collection of all $H$-entanglement breaking maps is a free mapping cone. Its realization is provided by \emph{all} $\Psi\in\mathrm{Pos}({\rm Psd}(H),{\rm Psd}(H))$, as is easily checked. In particular, its dual system is generated by all $\Psi\in\mathrm{Pos}({\rm Psd}(H),{\rm Psd}(H))$, i.e. it consists of those maps that admit a cp-factorization with some factor $\Psi\in\mathrm{Pos}({\rm Psd}(H),{\rm Psd}(H))$ in the middle.\qedhere
\end{enumerate}
\end{example}

Finally we prove, using our duality theorem, that two important free mapping cones are not finitely generated. 

\begin{theorem}[Non finitely generated mapping cones] Neither of the free mapping cones \begin{align*}\mathcal{POS}&=\left(\mathrm{Pos}({\rm Psd}(H_1),{\rm Psd}(H_2))\right)_{H_1,H_2}\\ \mathcal{CP}&=\left({\rm CP}(\mathsf B(H_1),\mathsf B(H_2))\right)_{H_1,H_2}\end{align*} is finitely generated. 
\end{theorem}
\begin{proof}
First assume that $\mathcal{POS}$ is finitely generated. Since we can take block-diagonal sums of positive maps, we can assume that it  is generated by a single element. In view of \cref{thm:dual} this means that $$\mathcal{POS}^\vee=\mathcal{EB}$$ has a concrete realization. This however implies there is some $\Psi\in \mathrm{Pos}({\rm Psd}(\tilde H_1),{\rm Psd}(\tilde H_2))$ such that any other map $\Phi\in\mathrm{Lin}(\mathsf B(H_1)_{\rm her},\mathsf B(H_2)_{\rm her})$ is entanglement breaking if and only if  $\Psi\otimes\Phi$ is positive. But $\Psi\otimes\Phi$ is already positive if $\Phi$ is only $\tilde H_1$-entanglement breaking, since $$\Psi\otimes\Phi=(\Psi\otimes{\rm id}_{\mathsf B(H_2)})\circ({\rm id}_{\mathsf B(\tilde H_1)}\otimes \Phi),$$ and  the first map ${\rm id}_{\mathsf B(\tilde H_1)} \otimes\Phi$ maps positive operators into separable operators.
Since there exist $\tilde H_1$-entanglement breaking maps that are not fully entanglement breaking (see e.g. Theorem 6 in \cite{Ch20}), the concrete realization provided by $\Psi$ is strictly larger than $\mathcal{EB},$ a contradiction.

Now assume $\mathcal{CP}$ is finitely generated, which means similarly that $\mathcal{CP}^\vee=\mathcal{CP}$ has a concrete realization by some map $\Psi\in \mathrm{Pos}({\rm Psd}(\tilde H_1),{\rm Psd}(\tilde H_2))$. Since all identity maps are cp, this implies that $ \Psi$ itself has to be cp. But then, for any $\tilde H_1$-positive map $\Phi$, $\Psi\otimes\Phi$ is clearly positive. Since there are maps which are $\tilde H_1$-positive but not cp (see for example \cite{Chru2009}), this shows that the concrete realization provided by $\Psi$ is strictly larger than $\mathcal{CP},$ a contradiction. 
\end{proof}

\begin{remark}
In \cite{sko} it was shown that the (classical)  mapping cone ${\rm Pos}({\rm Psd}_3,{\rm Psd}_3),$ which is the level $\mathcal{POS}(\C^3,\C^3)$ of our free mapping cone $\mathcal{POS}$, is not finitely generated as a (classical) mapping cone. It seems that the two results are independent, since finite generation of a single level classically is incomparable from finite generation of the full system as a free mapping cone.\end{remark}

\subsection{Conic group systems}\label{ssec:conic_groups} We now consider a refinement of the cone systems of \cref{ssec:conesystems} by adding group representations. We will show that the extension theorem for this stem results in a characterization of functionals that admit a positive extension which is invariant under the group action (\cref{cor:ext_conereps}). 

For a group $G$ let ${\tt FCone}_G$ denote the category of representations of $G$ on finite-dimensional convex cones. Its objects are triples $(\rho,C_\rho,V_\rho)$, often denoted just by $\rho$, where $(C_\rho,V_\rho)\in{\tt FCone}$ and  $$\rho\colon G\to {\rm Pos}(C_\rho):=\mathrm{Pos}(C_\rho,C_\rho)$$ is a homomorphism.   
A morphism between $\rho\colon G\to {\rm Pos}(C_\rho)$ and $\sigma\colon G\to {\rm Pos}(C_\sigma)$ is a positive intertwining map, i.e.\ a map $\Phi\in \mathrm{Pos}(C_\rho,C_\sigma)$ with $$\Phi\circ\rho(g)=\sigma(g)\circ \Phi$$ for all $g\in G$. Thus $ {\tt FCone}_G(\rho,\sigma)\subseteq\mathrm{Pos}(C_\rho,C_\sigma)$ is a convex cone.
The tensor product is defined as $$(\rho,C_\rho,V_\rho)\otimes (\sigma,C_\sigma, V_\sigma):=(\rho\otimes\sigma, C_\rho\omin C_\sigma, V_\rho\otimes V_\sigma)$$
and the dual representation is defined by 
\begin{align*}
\rho^*\colon G&\to {\rm Pos}(C_\rho^\vee);\ g \mapsto \rho(g^{-1})'.
\end{align*}    
Note that under the identification $V_\rho\otimes V_\sigma'=\mathrm{Lin}(V_\sigma,V_\rho)$ we have $${\rm Fix}(\rho\otimes\sigma^*)={\tt FCone}_G(\sigma,\rho),$$ 
where ${\rm Fix}(\rho) \subseteq V_\rho$ consists of the elements in $V_\rho$ that are mapped to themselves under $\rho$.

A short computation shows that  ${\tt FCone}_G$ is star autonomous, and  the functor 
$$\mathcal S\colon{\tt FCone}_G\to {\tt FVec};\ \rho\mapsto V_\rho, \  \Phi\mapsto\Phi$$
defines a valid stem. We call the arising $\mathcal S$-systems \emph{conic $G$-systems}. A conic $G$-system can be seen as a refinement of a cone system from the previous section. Indeed, when forgetting all levels indexed by a nontrivial action of $G$, one obtains a cone system.

For $\rho\in {\tt FCone}_G$ we have 
\begin{align*}
A(\rho)&=\{v\in V_\rho\mid f(v)\geqslant 0 \mbox{ for all $\rho^*$-invariant } f\in C_\rho^\vee\} \supseteq C_\rho\\
B(\rho)&=\{ c\in C_\rho\mid \forall g\in G\colon \rho(g)(c)=c \}={\rm Fix(\rho})\cap C_\rho\subseteq C_\rho.
\end{align*}

Now assume that $\rho$ is a direct sum of irreducible representations, which is for example always true if $G$ is finite. Let $V_\rho={\rm Fix}(\rho)\oplus W_\rho$ be the decomposition into the space on which $\rho$ acts trivially and its complement with respect to the irreducible decomposition.  Then $A(\rho)$ equals  the projection of  $C_\rho$ onto ${\rm Fix}(\rho),$ plus the whole of $W_\rho$.  
In particular, if $\rho\neq\mathbf 1$ is irreducible, then $B(\rho)=\{0\}$ and $A(\rho)=V_\rho.$ In particular, $\mathcal S$ is not self-dual. 

The extension theorem (\cref{thm:ext}) leads to the following group-invariant version of \cref{cor:riesz}:
\begin{corollary}[Extension theorem for conic $G$-systems]\label{cor:ext_conereps}
Let $\rho\colon G\to \mathrm{Pos}(C_\rho),\sigma\colon G\to \mathrm{Pos}(C_\sigma)$ be representations, with $C_\sigma$ sharp, and let $X\subseteq V_\rho$ be a subspace that intersects the interior of $C_\rho$. Then for a map $$\Psi\in \mathrm{Pos}(X\cap C_\rho,C_\sigma)$$ to  extend  to some   $\Phi\in \mathrm{Pos}(C_\rho,C_\sigma)$ that intertwines $\rho$ and $\sigma,$ it is necessary and sufficient that $\Psi\otimes{\rm id}_{V_\sigma'}$ maps $$A(\rho\otimes\sigma^*)\cap (X\otimes V_{\sigma}')$$ to $A(\sigma\otimes\sigma^*).$ 
\end{corollary}

For an intuition behind this corollary let us consider the following example. 

\begin{example}[Invariant extension of functionals under shear]
Consider the tuple $(\rho,C_\rho, \R^2)$ where $\rho$ is a representation $\rho\colon (\mathbb R, +)\to {\rm GL}(\mathbb R^2)$ defined by $$\rho(r)=\begin{pmatrix}
1 & r \\
0 & 1\end{pmatrix},$$ and $C_\rho$ is the upper half-plane. For $\sigma$ choose the trivial representation on $\mathbb R_{\geqslant 0}$ (i.e. the tuple $(\textrm{triv}, \Rnn, \R)$). For these two objects in ${\tt FCone}_G$ the intertwining maps are the $\rho$-invariant functionals, and $A(\sigma\otimes\sigma^*)=\mathbb R_{\geqslant 0}$. 
Now, $C_\rho^\vee$ corresponds to the positive $y$-axis, and all of these functionals are  $\rho^*$-invariant, so $A(\rho)=C_\rho$. By applying \cref{cor:ext_conereps} we find that every positive functional on $X$ admits a positive $\rho$-invariant extension to $\mathbb R^2.$ 
\end{example}

\subsection{Operator group systems}\label{ssec:operator_groups}
Instead of considering representations on a convex cone, we will now consider unitary representations on finite-dimensional Hilbert spaces, leading to a refinement of operator systems. We will again obtain the corresponding version of the extension theorem, leading to a unitary invariant extension. 

Given a group $G$, let us consider the (opposite) category ${\tt FHilb}_G^{\rm op}$ of unitary representations of $G$ on finite-dimensional Hilbert spaces. Its objects are tuples $(\rho,H_\rho)$ where $\rho$ is a group homomorphism $$\rho\colon G \to \mathsf U(H_\rho)$$ with $H_\rho\in {\tt FHilb}$ and $\mathsf U(H_\rho)$ denotes the group of unitary operators on $H_\rho$. A morphism from $\rho$ to $\sigma$ is a linear intertwining map  $T\in\mathsf B(H_\sigma,H_\rho)$ with $$\rho(g)T=T\sigma(g)$$ for all $g\in G$. Thus ${\tt FHilb}_G^{\rm op}(\rho,\sigma)\subseteq \mathsf B(H_\sigma,H_\rho)$ is a subspace. Tensor products and duals are defined as expected. The functor 
$$  \mathcal S\colon {\tt FHilb}_G^{\rm op}  \to {\tt FVec}; \  \rho\mapsto \mathsf B(H_\rho)_{\rm her},\ T\mapsto T^*\cdot T$$ defines a compact closed stem.
We call the arising $\mathcal S$-systems \emph{operator $G$-systems}. An operator $G$-system is a refinement of an operator system. After forgetting all levels indexed by a nontrivial action of $G$, one obtains an operator system. 
 
We find
\begin{align*}
A(\rho)&=\left\{ M\in\mathsf B(H_\rho)_{\rm her}\mid\forall h\in{\rm Fix}(\rho)\colon \langle Mh,h\rangle\geqslant 0\right\}=:{\rm Psd}(\rho)\\
    B(\rho)&=\left\{ \sum_i T_i^*T_i\mid T_i\colon H_\rho\to \mathbb C \ \rho\mbox{-invariant} \right\}=:{\rm Sos}(\rho),
\end{align*}
where ${\rm Psd}(\rho)\supseteq {\rm Psd}(H_\rho)$ because only the $h \in {\rm Fix}(\rho)$ set the restrictions instead of all $h \in H_\rho$, and ${\rm Sos}(\rho)$ consists of sums of squares of operators. In particular we have $B(\rho)\subseteq A(\rho)$, with strict inclusion in general. For example, if $\rho$ is irreducible, then $B(\rho)=\{0\}$ and $A(\rho)=\mathsf B(H_\rho)_{\rm her}.$

Since one can form direct sums of unitary representations, every free $G$-spectrahedron has a single concrete realization, and is thus of the form

$$\mc{G}(\sigma)=(\Psi\otimes{\rm id}_{\mathsf B(H_\sigma)})^{-1}({\rm Psd}(\rho\otimes\sigma))$$ for some $\Psi\colon X\to\mathsf B(H_\rho)_{\rm her}.$

Finally, the extension theorem (\cref{thm:ext}) reduces to an invariant version of Arveson's extension theorem in finite dimensions. 
\begin{corollary}[Extension theorem for operator $G$-systems]\label{cor:ext_opreps}
Let $X\subseteq \mathsf B(H_\rho)_{\rm her}$ be a subspace that intersects the interior of ${\rm Psd}(\rho),$ and let $\Psi\colon X \to \mathsf B(H_\sigma)_{\rm her}$ be a linear map. Then $\Psi$ admits an extension to a map $$\Phi=\sum_i T_i^*\cdot T_i\colon \mathsf B(H_\rho)_{\rm her}\to\mathsf B(H_\sigma)_{\rm her}$$ with $\sigma$-$\rho$ intertwining $T_i\colon H_\sigma\to H_\rho$ if and only if $\Psi\otimes {\rm id}_{\mathsf B(H_\sigma)'}$ maps $${\rm Psd}(\rho\otimes\sigma^*)\cap (X\otimes \mathsf B(H_\sigma)')$$ to ${\rm Psd}(\sigma\otimes\sigma^*).$
\end{corollary}

\subsection[Topological field theories]{Systems over topological field theories}\label{ssec:TFT}

We now describe topological quantum field theories as stems, and study the corresponding $\mc{S}$-systems. Let us consider, for a fixed $n \in \N$, the category ${\tt nCob}$, with oriented smooth $(n-1)$-dimensional manifolds as objects, and equivalence classes of cobordisms up to diffeomorphisms as morphisms. A \emph{cobordism} (\cref{fig:cobordism}) from one oriented smooth $(n-1)$-dimensional manifold $M$ to another one $N$ is an oriented  smooth $n$-dimensional manifold $\Sigma$ such that its boundary $\delta \Sigma$ is the disjoint union 
$$ \delta \Sigma = M \bigsqcup N.$$

\begin{figure}[t]
    \centering
    \includegraphics{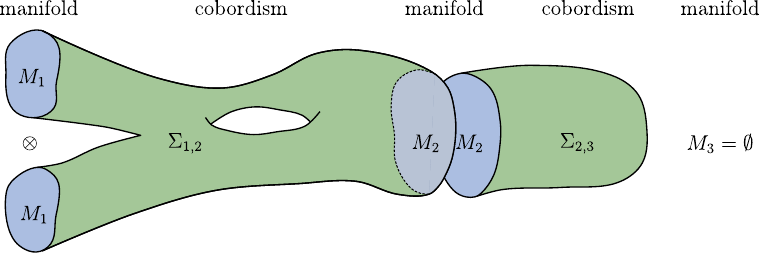}
    \caption{\small The two-dimensional smooth oriented manifolds $M_1\otimes M_1$, $M_2$ and the empty manifold $M_3 = \emptyset$ (in blue) are elements of ${\tt 2Cob}$. The cobordism $\Sigma_{1,2}$ is a $2+1$ dimensional manifold with boundary $M_1^{\otimes 2} \bigsqcup M_2$, and is a morphism in  ${\tt 2Cob}$; similarly for $\Sigma_{2,3}$. Composition of morphisms consists of gluing along a common boundary, so $\Sigma_{2,3}\circ \Sigma_{1,2}$ is a cobordism from $M_1^{\otimes 2}$ to $M_3$. The monoidal tensor product is the disjoint union of manifolds ($M_1^{\otimes 2}$ is an example), and the dual is obtained by swapping the orientation. }
    \label{fig:cobordism}
\end{figure}

${\tt nCob}$ is a symmetric monoidal category where composition is given by gluing two cobordisms along a common boundary, the tensor product is given by the disjoint union and the monoidal unit is the empty manifold $\emptyset$ embedded in $(n-1)$ dimensions. Moreover, it is a compact closed category, as shown in \cite{baez_2010}. The dual of a manifold $M$ is given by the same manifold with opposite orientation $M^*$ and a dual morphism (cobordism) has opposite direction or orientation. 

A stem on top of this category is a star autonomous functor $$\mc{S}\colon {\tt nCob} \rightarrow {\tt FVec}.$$ Such functors are known as \emph{topological (quantum) field theories} (TFT) \cite{Atiyah1988}.

Quantum field theories are a widely used physical framework to describe elementary particles as excitations of so-called quantum fields on a spacetime manifold. For \emph{topological} quantum field theories, the properties and evolution of the quantum fields do not depend on the metric of the spacetime. Apart from being of physical interest \cite{Carlip_2005}, these TFTs are of mathematical interest, because its metric independence makes all correlations functions a topological invariant of the spacetime manifold. 

How can we interpret the functor $\mc{S}$ as a TFT? First, remember that a finite-dimensional \emph{pure} quantum state is an element of a finite-dimensional vector space (when ignoring normalization). A TFT assigns to any manifold a vector space of possible pure quantum states that can exist on that space, and to any cobordism between two manifolds (spacetime) a linear map that describes the time evolution of these quantum states. But not all quantum states are pure; a general \emph{mixed} quantum state is an element from a psd cone, and every pure state $v \in \C^d$ is a mixed state of rank one given by $vv^* \in \textrm{Psd}_d \subseteq \textrm{Her}_d$. In the usual framework of TFTs, there is no notion of a psd cone, as all vectors in the assigned vector space correspond to allowed (unnormalized) quantum states, so this theory only describes \emph{pure} quantum states.  

There exist approaches to mixed state TFTs deviating from the original definition of Atiyah where the codomain of the functor is modified \cite{zini21}. We, however, stick to the original definition of a TFT as a stem. Our framework provides (co)intrinsic cones inside all vector spaces in the image, and we shall propose an interpretation where mixed states are elements from these cones. We will first focus on one-dimensional TFTs and then define pseudo-mixed state TFTs as systems on these stems. 


\subsubsection{One dimensional TFTs}
Let us restrict to the case $n=1$ and consider the stem $\mc{S}\colon {\tt 1Cob} \to {\tt FVec}$. The objects of ${\tt 1Cob}$ are zero-dimensional oriented manifolds consisting of points with either a positive or negative orientation, denoted $+$ and $-$. 
A star autonomous functor $\mc{S}\colon {\tt 1Cob} \to {\tt FVec}$ assigns a vector space to each point. It is fully determined by the choice for $\mc{S}(+) = V$, as $$\mc{S}(-) =  \mc{S}(+^*)= \mc{S}(+)'= V'.$$
All other objects $(k,l)$ in ${\tt 1Cob}$ are monoidal products (disjoint unions) of $k$ copies of $+$ and $l$ copies of $-$, for some $k,l \in \N$. Their images in ${\tt FVec}$ are of the form $$ \mc{S}(k,l) = V^{\otimes k} \otimes (V')^{\otimes l}.$$
The morphisms in ${\tt 1Cob}$ are one dimensional oriented manifolds with a boundary, i.e. arrows. The five connected components shown in \cref{fig:tft1} generate all of them. All other morphisms can be formed using the monoidal structure, i.e. taking disjoint unions and duals \cite{telebakovic2019}. The images under $\mc{S}$ are as follows: 
\begin{enumerate}[label=$(\roman*)$]
    \item The cobordism from $+$ to $+$ is mapped to the identity map $\textrm{id}_V\colon V \to V$.  
    \item The cobordism from $-$ to $-$ is mapped to the identity map $\textrm{id}_{V'}\colon V'\to V'$.
    \item The cobordism from $+ \bigsqcup -$ to $\emptyset$ is mapped to the evaluation map $ \textrm{ev}_{V}\colon V \otimes V'\to \R$.
    \item The cobordism from $\emptyset$ to $+ \bigsqcup -$ is mapped to the maximally entangled state $m_V\colon \R \to V \otimes V'$.
    \item The cobordism from $\emptyset$ to itself is mapped to $\textrm{dim}(V)\colon \R \to \R$, corresponding to multiplication with $\textrm{dim}(V)$. 
\end{enumerate} 

\begin{figure}[th]
    \centering
    \includegraphics{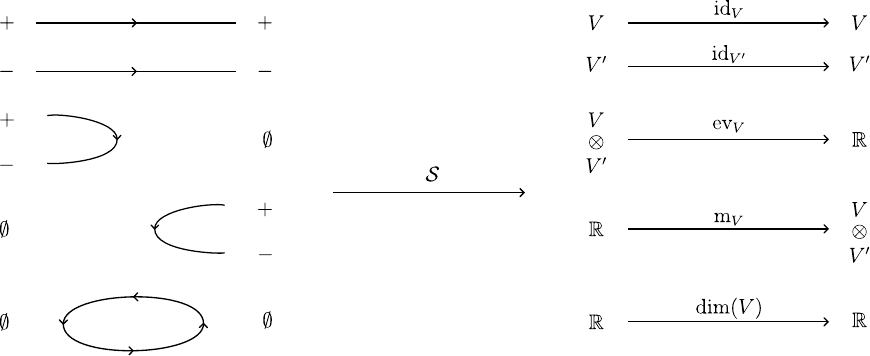}
    \caption{\small Objects in ${\tt 1Cob}$ consist of $+$ and $-$ points. The five arrows on the left represent the five generating morphisms in ${\tt 1Cob}$. Under a stem $\mc{S}$ (topological quantum field theory) these morphisms are mapped to the five corresponding linear maps in {\tt FVec}. The stem is fully determined by the choice of $\mc{S}(+) = V$.}
    \label{fig:tft1}
\end{figure}

There are morphisms between two objects $(k,l)$ and $(k',l')$ only in case $(k-l) = (k'-l')$. It follows that ${\tt 1Cob}$ is partitioned into disjoint subcategories 
$${\tt 1Cob} = \bigcup_{\substack{i = (k-l) \\ k,l\in\N}} {\tt 1Cob}_i$$ 
Of these subcategories only ${\tt 1Cob}_0$ is a compact closed subcategory, since it contains the monoidal unit $0$ ($k=l=0$), earlier denoted by $\emptyset$. All objects of ${\tt 1Cob}_0$ are of the form $k:= (k,k)\in \N$, and are self-dual since $k^* = k$. For all objects in ${\tt 1Cob}_i$ with $i \neq 0$ (i.e. $k\neq l$), the (co)intrinsic cones are trivial since there are no morphisms to the unit:
\begin{flalign*} 
A(k,l) &= V^{\otimes k} \otimes (V')^{\otimes l}\\
B(k,l)&= \{0\}.
\end{flalign*}
For this reason, we will restrict the stem to $$\mc{S}\colon {\tt 1Cob}_0 \to {\tt FVec}$$ and call it the TFT-stem. 

\begin{figure}[th]
    \centering
    \includegraphics{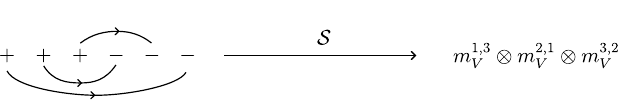}
    \caption{\small One of the $3!$ ways to pair three $+$ and three $-$ points. Seen as a morphism in ${\tt 1Cob}(0,3)$, the stem $\mc{S}$ maps it to a tensor product of three maximally entangled states on $V$, which is one of the extreme rays of $B(3)$. }
    \label{fig:tft_pairing}
\end{figure}

Let us describe the intrinsic and cointrinsic cones for the TFT-stem. The only type of morphism in ${\tt 1Cob}(0,k)$ creates $k$ paired $+$ and $-$ points. Its image is the $k$-th tensor power of maximally entangled states on $V \otimes V'$. There are $k!$ ways to pair the points (\cref{fig:tft_pairing}), so the cointrinsic cone is spanned by $k!$ vectors of the form $ m_{V}^{\otimes k}$, 
$$B(k) = \overline{\textrm{cone}}\{ m_{V}^{1,\sigma(1)} \otimes  m_{V}^{2,\sigma(2)} \otimes \ldots \otimes  m_{V}^{k,\sigma(k)} \mid \sigma \in S_k \},$$ where $S_k$ is the permutation group of a set of $k$ elements, and $m_V^{\alpha,\beta}$ is the maximally entangled state on the $\alpha$-th copy of $V$ and the $\beta$-th copy of $V'$.
Every of the generators of $B(k)$ is a sum of $d^k$ basis vectors of the full space $V^{\otimes k} \otimes (V')^{\otimes k}$. So clearly they are all conically independent and are thus extreme rays. So
$B(k)$ is a polyhedral cone with $k!$ extreme rays in $d^{2k}$ dimensions, with $d=\textrm{dim}(V)$. 
In addition, $B(k)$  is never full-dimensional inside $V^{\otimes k} \otimes (V')^{\otimes k},$ since there are many basis vectors that are never used to express the maximally entangled states, namely ones where there is no full pairing.
 
For the intrinsic cone we observe that the image of a map in ${\tt 1Cob}(k,0)$ is the $k$-th tensor power of the evaluation map, again for all different pairings of $V$ and $V'$. This leads to
$$A(k)= \bigcap_{\sigma \in S_k}(\textrm{ev}_V^{1,\sigma(1)} \otimes \ldots \otimes \textrm{ev}_V^{k,\sigma(k)})^{-1}(\Rnn). $$  
By duality $A(k)^\vee = B(k)$, so $A(k)$ is also polyhedral with $k!$ facets. Since $B(k)$ is not full, $A(k)$ is not sharp. Also, it is easy to check that $B(k) \subseteq A(k)$ (either directly or using \cref{lem:cones}). Finally, since $A(k) \omin A(l) \subsetneq A(k+l)$, assumption \eqref{eq:assumption} is not satisfied and there might not exist a tensor product of $\mc{S}$-systems for this stem. 

\subsubsection{TFT-systems and pseudo-mixed states}
Let us continue to study  $\mc{S}$-systems on the TFT-stem, and describe their relation to mixed quantum states. Let us fix $\mc{S}(+) = \mathrm{Her}_d$ for some $d\in \N$, defining a stem and thus a TFT. This TFT allows $k$ particles ($+$) and anti-particles ($-$) to be in a state corresponding to any vector in $\textrm{Her}_d^{\otimes k} \otimes \textrm{Her}_d^{\otimes k}$. An $\mc{S}$-system $\mathcal{G}$ on $\R$ is given by cones
$$\mathcal{G}(k) \subseteq \R \otimes \textrm{Her}_d^{\otimes k} \otimes \textrm{Her}_d^{\otimes k} = \textrm{Her}_{d^{2k}}.$$
Fixing the cone $\Rnn\subseteq \R$ at the base level $(k=0)$, all possible TFT-systems $\mathcal G$ over this cone  fulfill
$$B(k) = \underline{\mc{G}}_{\Rnn}(k) \subseteq \mc{G}(k) \subseteq \overline{\mc{G}}_{\Rnn}(k) = A(k).$$
If we interpret the cones $\mc{G}(k)$ as the cone of states that $k$ particles and $k$ anti-particles can be in, the various $\mc{S}$-systems can be seen as different mixed state theories for this particular TFT. 
\begin{definition}[$d$-dimensional pseudo-mixed TFT]
Let $\mc{S}$ be a TFT (stem) such that $\mc{S}(+) = \mathrm{Her}_d$ for some $d \in \N$. A \emph{$d$-dimensional pseudo-mixed TFT} is an $\mc{S}$-system over $\Rnn \subseteq \R$. 
\end{definition}

\begin{example}[$d$-dimensional pseudo-mixed TFTs]\label{ex:tft} \quad
\begin{enumerate}[label=$(\roman*)$]
    \item\label{tft1} One example of a $d$-dimensional pseudo-mixed TFT is $$\mc{G}(k) = \textrm{Psd}_{d^{2k}} \subseteq \textrm{Her}_{d^{2k}},$$ the cone of $2k$ particle mixed states in standard quantum theory, for particles that each live in a $d$-dimensional Hilbert space $\textrm{Her}_d$. 
    \item\label{tft2}  Another example is $\mc{G}(k) = B(k)$, only allowing the $2k$ particles to be in convex combinations of pairs of maximally entangled states between particles and anti-particles. The maximally entangled state $m_{\textrm{Her}_d} = \sum_{i,j=1}^d E_{ij} \otimes E_{ij}$ coincides with the physical notion of the maximally entangled quantum state of two particles in a $d$ dimensional Hilbert space (\cref{ex:max_ent}). 
Since no entanglement within the subset of particles (respectively anti-particles) is allowed in $B(k)$, this cone is a strict subcone of all physical states $\textrm{Psd}_{d^{2k}}$. 
    \item\label{tft3} Maximally, we could choose $\mc{G}(k) = A(k)$, containing all Hermitian matrices that are positive under pairs of evaluation maps. This choice yields a larger cone than the cone of physical mixed quantum states $\textrm{Psd}_{d^{2k}}$, hence our qualifier `pseudo' to the mixed TFT.\qedhere
\end{enumerate}
\end{example}

While for the pseudo-mixed TFT in \cref{ex:tft}\ref{tft2} and \ref{tft3} the asymmetry between particles and antiparticles is visible in the cones, this is not visible in \cref{ex:tft}\ref{tft1}. In the latter, the $k$ particles and $k$ anti-particles can be in any physical $2k$-particle state. 

There are two reasons why we interpret elements of a pseudo-mixed TFT as pseudo-mixed quantum states. First,
the fact that $B(k)$ is always contained in $\mc{G}(k)$ entails that when starting out with nothing (a vacuum) at level $0$ and creating particles, one can only create pairs of a particle and an anti-particle in a maximally entangled state. In view of the conservation laws of nature, this is a sensible requirement because the properties of the particle and anti-particle should be maximally anti-correlated. Similarly, $\mc{G}(k) \subseteq A(k)$ means that any particle/anti-particle pair can be annihilated by applying the evaluation map, resulting in a lower level of $\mc{G}$ and thus still in a valid pseudo-mixed state. This also makes sense physically, since if a particle and anti-particle annihilate, the remaining particles should be in a valid (pseudo-mixed) state. 
 
Secondly, for a pseudo-mixed TFT $\mathcal G$ that contains the psd cone at every level, all pure states, i.e.\ vectors in the image of another TFT with $V=\C^d$, can be seen as rank 1 pseudo-mixed states, since $xx^* \in \textrm{Psd}_{d^{2k}} \subseteq \mathcal G(k).$ So every pure state is an actual `physical' mixed state, and we can thus interpret the elements in $\mc{G}(k)$ as higher rank generalizations of pure states. Pseudo-mixed TFTs that do not contain the psd cone at every level will not contain all pure states, since the full psd cone is needed to capture all (possibly entangled) pure states. 

\section{Conclusion and outlook}\label{sec:outlook}

In this paper we have shown that the most important theorems about operator systems do not depend on the specific properties of the psd cones, but hold for far more general structures. We have generalized the underlying structure of Hermitian matrices and completely positive maps in operator systems to a functor from a star autonomous category to the category of finite-dimensional vector spaces, called stem (\cref{def:stem}). Choosing the category of finite-dimensional Hilbert spaces and the functor of \cref{ex:acs}\ref{aos}, we recover the example of operator systems. 

The psd cone is however still special in another way. Namely, it allows to explicitly construct a self-dual tensor product  $\otimes_{\textrm{sd}}$  as \begin{equation}\label{eq:sd}\textrm{Psd}_d \otimes_{\textrm{sd}} \textrm{Psd}_e \coloneq \textrm{Psd}_{de},\end{equation} for which, by self-duality of the psd cone,
$$(\textrm{Psd}_d \otimes_{\textrm{sd}} \textrm{Psd}_e)^\vee  = \textrm{Psd}_d^\vee \otimes_{\textrm{sd}} \textrm{Psd}_d^\vee. $$
Other than simplex cones, for which all tensor products coincide, we are not aware of any other cones for which such a tensor product `in the middle' of the minimal and maximal tensor product exists, as also mentioned in \cite{debruyn}. 

A self-dual tensor product for general convex cones would be extremely useful for generalizations of quantum theory, called general probabilistic theories (gpt), where states are modeled as elements from a convex cone $C$ \cite{Plavala_2023}. While in quantum theory a two-particle quantum state is modeled by an element from the cone in Eq.\eqref{eq:sd}, a similar notion does not exist for general convex cones, so, by lack of a better option, two-particle states are modeled by elements from the maximal tensor product $C \omax C$. We believe that two-sided systems (\cref{ssec:two-sided}) form a promising structure to this end, since any such system corresponds to a tensor product. In future work it would be interesting to connect the two-sided systems to entanglement between cones \cite{Aubrun21,Aubrun2022}.

There are more connections between general probabilistic theories and our framework. For example, the study of monogamy of entanglement between cones uses the family of all tensor powers of a cone \cite{Aubrun2022}, which can be seen as levels of a cone system (\cref{ssec:conesystems}). In addition, the incompatibility between measurements in quantum physics or general probabilistic theories is related to free spectrahedra and their inclusions \cite{Bluhm_2020,Bluhm_2022}. We believe that the framework of general conic systems is a natural structure to study these concepts. 
The question if all entanglement annihilating maps between two cones are entanglement breaking has been answered for Lorentz cones in \cite{aubrun_22}, making Lorentz cones an interesting class of gpts. It would be interesting to devise a Lorentz stem, with Lorentz cones as elements and positive maps between them as morphisms, to capture these results in our framework. 

\section*{Acknowledgements}
We want to thank Tobias Fritz for enlightening insights into category theory, as well as Michael Brennan and William Slofstra for helpful discussions about group operations.  GDLC and MVDE acknowledge support of the Stand Alone Project P33122-N of the Austrian Science Fund (FWF) and TN acknowledges support of the Stand Alone Project P36684 of the Austrian Science Fund (FWF). 

\bibliographystyle{abbrvnat}
\bibliography{references}
\end{document}